\documentclass[11pt,reqno]{amsart}
\usepackage{verbatim}
\usepackage{amsmath}
\usepackage{amsthm}
\usepackage{amsfonts}
\usepackage{mathrsfs}
\usepackage{amssymb}
\usepackage{amsbsy}
\usepackage{graphicx}
\usepackage{setspace}

\newtheorem{thm}{Theorem}[section]
\newtheorem{cor}[thm]{Corollary}
\newtheorem{lem}[thm]{Lemma}
\newtheorem{remark}[thm]{Remark}
\newcommand{\wc}{\overset{*}{\longrightarrow}}
\newcommand{\dist}{\mathrm{dist}}
\newcommand{\ext}{\mathrm{ext}}

\usepackage[left=.7in,right=.7in,top=.7in,bottom=.7in,footnotesep=18pt]{geometry}

\begin{document}

\title[Asymptotics of Carleman polynomials]{An asymptotic integral representation for Carleman orthogonal
polynomials}
\author[E.
Mi\~{n}a-D\'{\i}az]{Erwin Mi\~{n}a-D\'{\i}az}
\date{}
\address{Indiana University-Purdue University Fort Wayne, Department of Mathematical Sciences,
2101 E. Coliseum Blvd., Fort Wayne, IN
46805-1499, USA. Email: minae@ipfw.edu.}

\begin{abstract}
In this paper we investigate the asymptotic
behavior of polynomials that are orthonormal over the
interior domain of an analytic Jordan curve $L$
with respect to area measure. We prove that,
inside $L$, these polynomials behave asymptotically like a sequence of certain
integrals involving the canonical conformal map
of the exterior of $L$ onto the exterior of the
unit circle and certain meromorphic kernel
function defined in terms of a conformal map of
the interior of $L$ onto the unit disk. The
error term in the integral representation is
proven to decay geometrically and sufficiently
fast, allowing us to obtain more precise
asymptotic formulas for the polynomials under
certain additional geometric considerations.
These formulas yield, in turn, fine results on
the location, limiting distribution and
accumulation points of the zeros of the
polynomials.
\end{abstract}

\keywords{Orthogonal polynomials; asymptotic behavior; integral representation; zeros of
polynomials; equilibrium measure; Schwarz reflection
principle; conformal map; lemniscate.}

\maketitle

\section{Introduction}

Let $L_1$ be an analytic Jordan curve in the
complex plane $\mathbb{C}$ and let $\Omega_1$ and
$G_1$ be, respectively, the unbounded and bounded
components of $\overline{\mathbb{C}}\setminus L_1$. Let
$\{P_n(z)\}_{n=0}^\infty$ be the unique sequence
of polynomials orthonormal over $G_1$, i.e.,
satisfying
\begin{equation}\label{eq63}
P_n(z)=\kappa_nz^n+\mathrm{lower\ degree\
terms},\quad \kappa_n>0,
\end{equation}
\begin{equation}\label{eq64}
\frac{1}{\pi}\int_{G_1}P_{n}(z)\overline{P_{m}(z)}dA(z)=\left\{\begin{array}{ll}
                                                                 0,\ &\ n\not=m, \\
                                                                 1,\ &\ n=m,
                                                               \end{array}
\right.
\end{equation}
where $dA$ denotes the area measure.

These polynomials were first studied
 by Carleman
\cite{Carleman} in 1922, who proved that they
satisfy the strong asymptotic formula
\[
P_{n}(z)=\sqrt{n+1}\,\phi'(z)[\phi(z)]^n\left[1+o(1)\right]
\]
locally uniformly as $n\to\infty $ on certain
open set $\Omega_\rho\supset\overline{\Omega}_1$, where
$\phi$ is the conformal map of $\Omega_1$ onto
the exterior of the unit circle satisfying that
$\phi(\infty)=\infty$, $\phi'(\infty)>0$ (see
Theorem \ref{thm5} below for more details). In particular, it follows that the zeros of $P_n$ must accumulate as $n\to\infty$ on $\mathbb{C}\setminus \Omega_\rho$.

However, despite the fact that Carleman's work
pioneered the study of polynomials orthogonal
over planar domains, and that several subsequent
works have been devoted to the subject (see,
e.g., \cite{Korovkin}, \cite{Smirnov}, \cite{Suetin1},
 \cite{victor}, \cite{Levin},
\cite{mina}), no significant progress has been made in
understanding the behavior of the polynomials
$P_n(z)$ and its zeros in the orthogonality
domain $G_1$ (more exactly, in $\mathbb{C}\setminus\Omega_\rho$). The aim of this paper is precisely
to clarify this fundamental question.

The precise statements of our results are
contained in Section \ref{mainresults} below.
Roughly speaking, we show that the behavior of
$P_n$ inside $G_1$ is governed by both the
exterior map $\phi$ and certain meromorphic ``kernel
function" $L(\zeta,z)$, which is defined in
terms of a conformal map $\varphi$ of $G_1$ onto
the unit disk by
\[
L(\zeta,z):=\frac{\varphi'(\zeta)\varphi'(z)}{[\varphi(\zeta)-\varphi(z)]^2}\,.
\]

More exactly, $\phi$ and $L(\zeta,z)$ canalize
their influence over $P_n$ through the asymptotic
integral representation
\begin{equation}\label{eq65}
P_{n}(z)=\frac{(n+1)^{-1/2}}{2\pi
 i}
 \oint_{L_1}L(\zeta,z)[\phi(\zeta)]^{n+1}d\zeta
+o(1)\,,\quad z\in G_1.
\end{equation}
The full version of (\ref{eq65}), that is, with a
good estimate on the rate of decay of the error
term, is stated as Theorem \ref{thm3} of
Subsection \ref{asymptoticformulas}.

We then exploit such a representation to obtain
very detailed asymptotics of $P_n$ (see Theorems
\ref{thm1} and \ref{thm2} of Subsection
\ref{asymptoticformulas}) valid for certain quite
general sets  $\Omega_\rho$ having piecewise
analytic boundary ($\Omega_\rho$ is the set on
which Carleman's formula holds). Of particular
interest is Theorem \ref{thm1}, which depicts the
oscillatory behavior of $P_n$ in the interior of
$\mathbb{C}\setminus\Omega_\rho$.  As a
consequence, fine results on the location,
limiting distribution and accumulation points of
the zeros of the polynomials follow. These are
discussed in Subsection \ref{thezeros}.

In Subsection \ref{lemniscates} we examine in
detail the case of $L_1$ being the lemniscate
$\{z:|z^s-1|=R^s\}$, where $s\geq 2$ is an
integer and $R>1$. This example illustrates the
situation in which Theorem \ref{thm1} fails to
describe the behavior of certain subsequences of
$\{P_n\}_{n\geq 0}$. In particular, it
exemplifies what could happen if a key hypothesis
in the statement of some of the zero results
fails to hold true. Finally, the proof of all the
results are given in Section \ref{proofs}.

It is important to remark that the results
obtained in this paper have their counterparts
for polynomials orthogonal over an analytic curve
with respect to a positive analytic weight. These
will appear in a paper to be submitted for
publication soon.

We also remark that another important system of
polynomials associated with a Jordan curve, the
so-called Faber polynomials (see the monograph
\cite{suetinfaber}), also satisfies an integral
representation similar (though simpler and exact)
to (\ref{eq65}). In a recent paper
\cite{minafaber}, the author has exploited such a
representation to derive precise asymptotic
formulas for the Faber polynomials associated to
a piecewise analytic curve. The results and
techniques of proof in \cite{minafaber} exhibit a
resemblance to those of the present paper that is
worth noticing, since Faber polynomials have been
often used as a tool for obtaining asymptotic
properties of orthogonal polynomials over curves
and regions.

\section{Main results}\label{mainresults}

\subsection{Asymptotic formulas}\label{asymptoticformulas}

The following notation will be used throughout the
paper. Given $r\geq 0$, we set
\[
\mathbb{T}_r:=\{w:|w|=r\},\quad
\mathbb{E}_r:=\{w:r<|w|\leq \infty\},\quad
\mathbb{D}_r:=\{w:|w|<r \}.
\]
If $K$ is a set and $f$ is a function defined on
$K$, $\overline{K}$ denotes the closure of $K$ and
$f(K):=\{f(z):z\in K\}$.

As in the introduction, $L_1$ is an analytic
Jordan curve in the complex plane and $\Omega_1$
and $G_1$ are, respectively, the unbounded and
bounded components of $\overline{\mathbb{C}}\setminus
L_1$.

Let $\psi(w)$ be the unique conformal map of
$\mathbb{E}_1$ onto $\Omega_1$ satisfying that
$\psi(\infty)=\infty$, $\psi'(\infty)>0$. Let
$\rho\geq 0$ be the smallest number such that
$\psi$ has an analytic and \emph{univalent}
continuation from $\mathbb{E}_1$ to
$\mathbb{E}_\rho$. Because $L_1$ is analytic,
$\rho<1$. For every $r\in [\rho,\infty)$, let
\[
\Omega_r:=\psi(\mathbb{E}_r),\quad L_r:=\partial
\Omega_r,\quad G_r:=\mathbb{C}\setminus
\overline{\Omega}_r,
\]
and let $\phi(z):\Omega_\rho\mapsto
\mathbb{E}_\rho$ be the inverse of $\psi$.
Observe that, for every $r>1$, $L_r$ is an
analytic Jordan curve.

The polynomials $P_n(z)$, $n=0,1,2,\ldots,$ that
are orthonormal over $G_1$ with respect to area
measure, that is, satisfying (\ref{eq63}) and
(\ref{eq64}), were first considered by Carleman
in \cite{Carleman}. We will refer to them as the
Carleman polynomials\footnote{Carleman polynomials are also and often called Bergman polynomials.} for the curve $L_1$ (or for
the domain $G_1$).

\begin{thm}\label{thm5} \emph{(Carleman \cite{Carleman})} The following asymptotic formulas hold true:
\begin{equation*}
\kappa_n=\sqrt{n+1}\,[\phi'(\infty)]^{n+1}\left[1+\mathcal{O}\left(\rho^{2n}\right)\right],
\end{equation*}
\begin{equation}\label{eq77}
P_{n}(z)=\sqrt{n+1}\,\phi'(z)[\phi(z)]^n\left[1+h_{n}(z)\right], \qquad z\in \Omega_\rho,
\end{equation}
where $h_n(z)$ converges uniformly to zero as $n\to\infty$ on each $L_r$, $\rho<r<\infty$, with the following rate:
\begin{equation}\label{eq38}
h_n(z)=\left\{\begin{array}{ll}
\mathcal{O}\left(\rho^n\right),\ &\ r> 1, \\
         \mathcal{O}\left(\sqrt{n}\rho^n\right),\ &\ r= 1, \\
         {\displaystyle \mathcal{O}\left(n^{-1/2}(\rho/r)^n\right)},\ &\ \rho<r<1.
       \end{array}\right.
\end{equation}
\end{thm}

Note that, by the maximum modulus principle for
analytic functions, the estimates given for
$h_n(z)$ on $L_r$ are indeed valid on
$\overline{\Omega}_r$. Carleman \cite[Satz
IV]{Carleman} stated and proved this theorem as
holding uniformly in the exterior of the curve
$L_1$ with the estimate
$h_n(z)=\mathcal{O}\left(\sqrt{n}\rho^n\right)$.
However, without any variation, his proof equally
yields Theorem \ref{thm5} in the way it has been
stated above (see the first paragraph of
Subsection \ref{subsection3.1} below preceding
the proof of Theorem \ref{thm3}).

So far as the author can learn, Carleman's
formula is the only known result of substantial
generality that neatly describes the asymptotic
behavior of the polynomials $P_n$.  Notice that,
since both $\phi$ and $\phi'$ do not vanish on
$\Omega_\rho$, Theorem \ref{thm5} implies that
the zeros of $P_n$ must accumulate, in the limit,
on $\mathbb{C}\setminus\Omega_\rho$. That is to say,  any closed subset of $\Omega_\rho$ contains zeros of at most finitely many polynomials $P_n$. Our
investigation focuses precisely in understanding
how Carleman polynomials and their zeros behave
in $\mathbb{C}\setminus\Omega_\rho$.

We now discuss a well-known important relation
between Carleman polynomials and the conformal
maps of $G_1$ onto the unit disk $\mathbb{D}_1$. Let $f$ be a
function in the Bergman space
$\mathcal{B}^2(G_1)$ of the domain $G_1$, that
is, $f$ is and analytic function defined on $G_1$
satisfying
\[
\frac{1}{\pi}\int_{G_1}|f(\zeta)|^2dA(\zeta)<\infty,
\]
and let $\tau=\tau(f)\geq 1$ be the largest
number such that $f$ has an analytic continuation
to $G_\tau$. It is well-known \cite[pp.
128-131]{Walsh} that
\begin{equation}\label{eq62}
f(\zeta)=\sum_{k=0}^\infty \alpha_k(f)P_k(\zeta),
\quad \zeta\in G_\tau,
\end{equation}
where
\[
\alpha_k(f)=\frac{1}{\pi}\int_{G_1}f(\zeta)\overline{P_k(\zeta)}dA(\zeta),
\quad k=0,1,2,\ldots,
\]
and the series in the right-hand side of
(\ref{eq62}) converges locally uniformly to $f$
on $G_\tau$.

Let us then apply this result to the so-called
Bergman kernel function $K(\zeta,z)$ of the space
$\mathcal{B}^2(G_1)$ (see \cite[pp.
30-33]{Gaier1}), which has the reproducing
property
\begin{equation}\label{eq68}
f(z)=\frac{1}{\pi}\int_{G_1}f(\zeta)\overline{K(\zeta,z)}dA(\zeta)
\quad \forall\,f\in \mathcal{B}^2(G_1),\ z\in
G_1,
\end{equation}
and can be expressed in terms of a conformal map
$\varphi$ of $G_1$ onto $\mathbb{D}_1$ as
\begin{equation}\label{eq67}
K(\zeta,z)=\frac{\overline{\varphi'(z)}\varphi'(\zeta)}{\left[1-\overline{\varphi(z)}\varphi(\zeta)\right]^2}\,,\quad
\zeta,\, z\in G_1.
\end{equation}

Because $L_1$ is a Jordan curve, any such map
$\varphi$ can be extended as a continuous and
bijective function $\varphi:\overline{G}_1\to
\overline{\mathbb{D}}_1$. Moreover, being $L_1$
analytic, $\varphi$ has a meromorphic
continuation to $G_{1/\rho}$, which is indeed
given by
\begin{equation}\label{eq66}
\varphi(z)=\frac{1}{\overline{\varphi\left(z^*\right)}},
\qquad z\in G_{1/\rho}\setminus \overline{G}_1\,,
\end{equation}
where
\begin{equation}\label{eq70}
z^*=\psi\left(1\big/\overline{\phi(z)}\right), \quad
z\in G_{1/\rho}\cap\Omega_\rho,
\end{equation}
is the Schwarz reflection  about $L_1$ of any
point $z\in G_{1/\rho}\cap\Omega_\rho$ (see
\cite{Davis} for details).

It follows from these considerations and relation
(\ref{eq67}) that for every $z\in G_1$,
$K(\cdot,z)$ is analytic on $\overline{G}_1$, and since
by (\ref{eq68}),
$\alpha_k(K(\cdot,z))=\overline{P_k(z)}$, we obtain
from (\ref{eq62}) that
\begin{equation}\label{eq71}
\frac{\overline{\varphi'(z)}\varphi'(\zeta)}{\left[1-\overline{\varphi(z)}\varphi(\zeta)\right]^2}=\sum_{k=0}^\infty
\overline{P_k(z)}P_k(\zeta), \quad \zeta\in \overline{G}_1,\
z\in G_1,
\end{equation}
with the series converging uniformly in $\zeta\in
\overline{G}_1$ for each fixed $z\in G_1$.

As we shall see soon, it turns out that the
behavior of $P_n$ in $G_1$ is closely related to
a symmetric ``meromorphic kernel'' $L(\zeta,z)$
that well resembles the Bergman kernel, namely, the function of two
variables
\begin{equation}\label{eq69}
L(\zeta,z):=\frac{\varphi'(\zeta)\varphi'(z)}{[\varphi(\zeta)-\varphi(z)]^2},\quad
\zeta,\,z\in G_1.
\end{equation}

That $L(\zeta,z)$ does not depend on the choice
of the conformal map $\varphi$ can be easily
established from the fact that any two conformal
maps $\varphi$ and $\varphi_1$ of $G_1$ onto
$\mathbb{D}_1$ are related through a M\"obius
transformation, that is,
\[
\varphi(z)=e^{i\theta}\frac{\varphi_1(z)-\varphi_1(z_0)}{1-\overline{\varphi_1(z_0)}\varphi_1(z)},\qquad
e^{i\theta}=\varphi'(z_0)\left(1-|\varphi_1(z_0)|^2\right)/\varphi'_1(z_0),
\]
where $z_0$ is that point of $G_1$ mapped by
$\varphi$ onto $0$.

If we specifically choose a map $\varphi$ that
does not vanish on $G_1\cap\Omega_\rho$, then in
view of (\ref{eq66}) and (\ref{eq70}), $\varphi$
has a one-to-one analytic continuation to
$G_{1/\rho}$, so that $L(\zeta,z)$ can be
extended as a continuous function
\[
L(\zeta,z):G_{1/\rho}\times G_{1/\rho}\to
\overline{\mathbb{C}}
\]
such that for every fixed $z\in G_{1/\rho}$,
$L(\cdot,z)$ is analytic on $G_{1/\rho}\setminus
\{z\}$ with a Laurent expansion at $z$ of the
form
\[
L(\zeta,z)=\frac{1}{(\zeta-z)^2}+a_0+a_1(\zeta-z)+a_2(\zeta-z)^2+\cdots\,.
\]

We combine Theorem \ref{thm5} with relation
(\ref{eq71}) to deduce the following theorem:

\begin{thm}\label{thm3} With the notations above, we
have that
\begin{equation}\label{eq2}
 P_{n}(z)=\frac{(n+1)^{-1/2}}{2\pi
 i}
 \oint_{L_1}L(\zeta,z)[\phi(\zeta)]^{n+1}d\zeta
 +\epsilon_n(z),\quad z\in G_1,
\end{equation}
where the functions $\epsilon_n(z)$,
$n=0,1,2,\ldots$, are analytic on $G_{1/\rho}$
and have the following property: if $E\subset
G_{1/\rho}$ is such that for some
$0\leq\tau<1/\rho$,
\begin{equation*}
P_n(z)=\mathcal{O}\left(\sqrt{n}\tau^n\right)
\end{equation*}
uniformly on $E$ as $n\to\infty$, then
\begin{equation*}
\epsilon_n(z)=\mathcal{O}\left(\sqrt{n}(\tau\rho)^n\right)
\end{equation*}
uniformly on $E$ as $n\to\infty$.
\end{thm}

A simple consequence of Theorem \ref{thm3} is the
following improvement in Carleman's formula
regarding the error estimate:
\begin{cor}\label{cor1} Indeed, for the functions $h_n(z)$ defined by (\ref{eq77}), we have that $
h_n(z)=\mathcal{O}(\rho^n)$ uniformly in $z\in L_1$ as $n\to\infty$.
\end{cor}

From Theorem \ref{thm5} and the maximum modulus
principle for analytic functions we see that if
$\rho<r<1$, then
$P_n(z)=\mathcal{O}(\sqrt{n}r^n)$ uniformly in
$z\in \overline{G}_r$ as $n\to\infty$.
Consequently, it follows from Theorem \ref{thm3}
after integrating by parts over $L_1$ and making
the change of variables $\zeta=\psi(t)$ that
\begin{eqnarray}\label{eq28}
P_{n}(z)&=& \frac{\sqrt{n+1}\,\varphi'(z)}{2\pi
 i} \oint_{L_1}
 \frac{\phi'(\zeta)[\phi(\zeta)]^{n}d\zeta}{\varphi(\zeta)-\varphi(z)}+\epsilon_n(z)\nonumber\\
&=&\frac{\sqrt{n+1}\,\varphi'(z)}{2\pi
 i}\oint_{\mathbb{T}_1}\frac{t^{n}dt}{\varphi(\psi(t))-\varphi(z)}
+\mathcal{O}\left(\sqrt{n}(r\rho)^n\right),\quad
z\in\overline{G}_r, \quad \rho<r<1.
\end{eqnarray}

To illustrate how effectively Theorem \ref{thm3}
can be exploited for deriving finer asymptotic
results, let us consider the situation in which
the boundary $L_\rho$ of $\Omega_\rho$ is a
piecewise analytic curve without cusps or smooth
corners (see, however, Remark \ref{rem1} below).
More precisely, we shall assume $\psi$ satisfies
Conditions A.1 and A.2 to be stated in what
follows.

We define an \emph{analytic arc} as being the
image of the interval $[0,1]$ by a function
$f(t)$ analytic in $[0,1]$ such that
$f'(t)\not=0$ for all $t\in [0,1]$ and
$f(t_1)\not=f(t_2)$ for all $0<t_1<t_2< 1$. The
endpoints of the arc are $f(0)$ and $f(1)$, which
may coincide. We call the arc \emph{simple} if
$f$ is one-to-one on $[0,1]$. Notice that,
according to this definition, an analytic Jordan
curve is also an analytic arc. Our first
assumption is:

\begin{enumerate}
\item[\textbf{A.1:}] The map $\psi$ has a continuous extension to $\overline{\mathbb{E}}_\rho$
and there are $s\geq 1$ distinct points
$\omega_1,\omega_2,\ldots,\omega_s$ in $\mathbb{T}_\rho$
such that if $w_1\not=w_2$ are two points of $
\mathbb{T}_\rho$ for which $\psi(w_1)=\psi(w_2)$,
then $\{w_1,w_2\}\subset \{\omega_1,\ldots,\omega_s\}$.
Moreover, if $\ell$ is any of the $s$ open
circular arcs that compose
$\mathbb{T}_\rho\setminus
\{\omega_1,\omega_2,\ldots,\omega_s\}$, say with endpoints
$\omega_k$, $\omega_j$, then
$\psi\left(\overline{\ell}\right)$ is an analytic arc
with endpoints $\psi(\omega_k)$, $\psi(\omega_j)$ (see
Figure \ref{fig2} below).
\end{enumerate}

Thus, $L_\rho=\partial \Omega_\rho$ is a  piecewise analytic
curve. Let $z\in L_\rho$ and
$w=\rho e^{i\Theta}$ be such that $z=\psi(w)$. The
\emph{exterior angle at $z$ relative to $w$} is
defined to be that angle $\alpha\in [0, 2\pi]$
such that
\[
\arg\left[\psi\left(\rho e^{i\theta}\right)-z\right]\to\left\{\begin{array}{ll}
                                   \beta &\ \,\mathrm{as}\ \theta\to \Theta-\,,\\
                                   \beta+\alpha &\ \,\mathrm{as}\
                                   \theta\to\Theta+\,.
                                 \end{array}
\right.
\]

Let
\[
z_k:=\psi(\omega_k), \quad k\in\{1,2,\ldots,s\},
 \]
be ``the corners of $L_\rho$". Notice that they are not necessarily pairwise distinct.

For each $ k\in\{1,2,\ldots,s\}$, let
$\lambda_k\in [0,2]$ be such that $\lambda_k\pi$
is the exterior angle at $z_k$ relative to
$\omega_k$. Our second
assumption on $\psi$ is:
\begin{enumerate}
\item[\textbf{A.2:}] $\lambda_k\not\in \{0,1,2\}$ for every $
k\in\{1,2,\ldots,s\}$ (i.e., $L_\rho$ has no
cusps or smooth corners).
\end{enumerate}

\begin{figure}
\centering
\includegraphics[scale=.55]{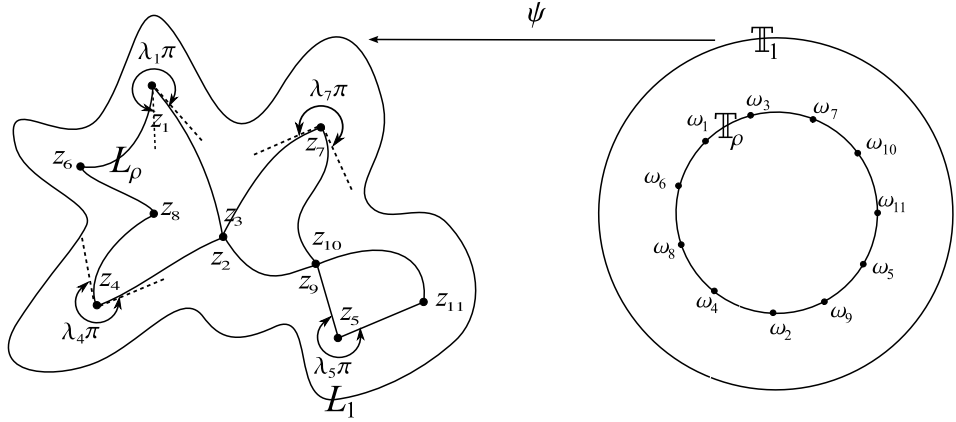}
\caption{Illustration of a map $\psi$ satisfying
Conditions A.1 and A.2.}\label{fig2}
\end{figure}

We assume that the $\omega_k$'s have been indexed in
such a way that for some $u\in\{1,2,\ldots,s\}$,
\[
\lambda_1=\lambda_2=\cdots=\lambda_u<\lambda_{u+1}\leq\cdots\leq
\lambda_s\,.
\]

Let
$\Theta_1,\Theta_2,\ldots,\Theta_s$ be the
arguments of the numbers $\omega_k$, that is,
\[
\omega_k=\rho e^{i\Theta_k},\quad 0\leq\Theta_k<2\pi,\quad 1\leq
k\leq s\,.
\]
By a well-known result of Lehman (see
\cite{Lehman} and Subsection
\ref{lehmanexpansion} below), the map $\psi(w)$
admits an asymptotic expansion about each
$\omega_k$. In particular, for each
$k\in\{1,2,\ldots,s\}$, the limit
\begin{equation*}
A_k:=\lim_{\underset{w\in\mathbb{E}_\rho}{w\to
\omega_k}}\frac{\psi(w)-z_k}{(w-\omega_k)^{\lambda_k}}
\end{equation*}
is a complex number different from zero.
Obviously, the value of this number $A_k$ depends
on the branch chosen for the function
$(w-\omega_k)^{\lambda_k}$ in a
$\delta$-neighborhood of the form
$\{w\in\mathbb{E}_\rho:0<|w-\omega_k|<\delta\}$.
Here we choose the one corresponding to the
branch of the argument
\[
\Theta_k-\pi<\arg(w-\omega_k)<\Theta_k+\pi,\quad
w\in \mathbb{C}\setminus\{t\omega_k:t\leq 1\}.
\]

The symbol $\binom{a}{b}$ stands for the generalized
binomial coefficient, i.e.,
\[
\binom{a}{b}:=\frac{\Gamma(a+1)}{\Gamma(b+1)\Gamma(a-b+1)},
\]
where $\Gamma$ denotes the Euler gamma function.

The behavior of $P_n$ inside $G_\rho$ is given
first. It strongly reflects the dependance of
$P_n$ on both $\phi$ and $L(\zeta,z)$.

\begin{thm}\label{thm1} For every $z\in G_\rho$,
we have
\begin{equation}\label{eq56}
\frac{P_n(z)}{\sqrt{n+1}\binom{n}{
-\lambda_1-1}\rho^{n+1+\lambda_1}}
=-\sum_{k=1}^u
L(z_k,z)A_ke^{i(n+1+\lambda_1)\Theta_k} +R_n(z),
\end{equation}
where
\begin{equation}\label{eq12}
R_n(z)=
 \left\{\begin{array}{ll}
                      \mathcal{O}\left(n^{-\lambda_1}\right),\ &\mathrm{if}\  0<\lambda_1<1,\ \lambda_1\not=1/2,\\
                      \mathcal{O}\left(n^{-1}\log n\right),\ &\mathrm{if}\ \lambda_1=1/2, \\
                      \mathcal{O}\left(n^{-1}\right),\ &\mathrm{if}\ 1<\lambda_1<2, \\
                      \end{array}\right.+
\left\{\begin{array}{ll}
\mathcal{O}\left(n^{\lambda_1-\lambda_{u+1}}\right),\
&\mathrm{if}\ u<s, \\
0,\ &\mathrm{if}\ u=s,
\end{array}\right.
\end{equation}
uniformly as $n\to\infty$ on compact subsets of
$G_\rho$.
\end{thm}

To state the behavior of $P_n(z)$ for points $z$ near $L_\rho$ we need the following piece of notation. For $\delta\in (0,\rho)$ and $1\leq k\leq s$, we
define the following open ``pie-slices":
\[
\Sigma_{\delta,k}:=\left\{w:\rho-\delta<|w|<\rho^2/(\rho-\delta),\
\Theta_k-\delta<\arg(w)<\Theta_k+\delta \right\}.
\]

\begin{thm}\label{thm2} (a) For every $\delta\in (0,\rho)$, there exists $\sigma\in (\rho-\delta,\rho)$
such that  $\psi$ has a one-to-one analytic
continuation to $\mathbb{E}_\sigma\setminus
\cup_{k=1}^s\overline{\Sigma_{\delta,k}}$, and if
\[
V_{\delta,\sigma}:=\left\{z=\psi(w):w\in \mathbb{E}_\sigma\setminus \cup_{k=1}^s\overline{\Sigma_{\delta,k}}\right\}
\]
and $\phi$ denotes the inverse of $\psi$, then
\begin{equation}\label{eq61}
\frac{P_n(z)}{\sqrt{n+1}}=\phi'(z)[\phi(z)]^n-\binom{n}{
-\lambda_1-1}\rho^{n+1+\lambda_1}\left(\sum_{k=1}^u
L(z_k,z)A_ke^{i(n+1+\lambda_1)\Theta_k}
+R_n(z)\right),\quad z\in V_{\delta,\sigma}\cap G_1,
\end{equation}
where $R_n(z)$ satisfies (\ref{eq12}) locally
uniformly on $V_{\delta,\sigma}\cap G_1$ as
$n\to\infty$.

(b) For every corner $z_j$,
\begin{equation}\label{eq43}
\frac{P_n(z_j)}{\sqrt{n+1}\binom{n}{
\lambda^*_j-1}\rho^{n+1-\lambda_j^*}}=
\sum_{\underset{\lambda_k=\lambda^*_j}{k\,:\,
z_k=z_j}
}(A_k)^{-1}e^{i(n+1-\lambda_j^*)\Theta_k} +o(1),
\end{equation}
where $\lambda^*_j=\max\{\lambda_k: z_k=z_j,\
1\leq k\leq s\}$.
\end{thm}

The rate of decay of the $o(1)$-error term in
(\ref{eq43}) can be estimated by comparing the
terms in the right-hand side of (\ref{eq59}).

\begin{remark}\label{rem1}\emph{Totally similar results can be obtained for
considerably more general piecewise analytic
curves $L_\rho$, including those having inner
cusps (as viewed from $\Omega_\rho$) and smooth
corners, more specifically, under the following
assumptions:
\begin{enumerate}
\item[A.1':] The map $\psi$ has a continuous extension to $\overline{\mathbb{E}}_\rho$
and there are $s\geq 1$ distinct points
$\omega_1,\omega_2,\ldots,\omega_s$ in $\mathbb{T}_\rho$
such that if $\ell$ is any of the $s$ open
circular arcs that compose
$\mathbb{T}_\rho\setminus
\{\omega_1,\omega_2,\ldots,\omega_s\}$, say with endpoints
$\omega_k$, $\omega_j$, then $\psi$ is one-to-one on
$\ell$ and $\psi\left(\overline{\ell}\right)$ is an
analytic arc with endpoints $\psi(\omega_k)$,
$\psi(\omega_j)$.
\item[A.2':] $\lambda_k>0$ for every $ k\in\{1,2,\ldots,s\}$, and if
$\lambda_k\in\{1,2\}$ for all $
k\in\{1,2,\ldots,s\}$, then there is at least one
$k$ for which logarithmic terms occur in the
Lehman expansion of $\psi$ about $\omega_k$.
\end{enumerate}
The statements of the corresponding results and
their proofs are, however, more cumbersome, and
so we have sacrificed generality in the present
paper for the benefit of clarity. The interested
reader will find useful to consult
\cite{minafaber}, where similar asymptotic
formulas have been derived for the Faber
polynomials of a domain $\Omega_\rho$ satisfying
A.1' and A.2' by exploiting an integral
representation for these polynomials that is
somewhat similar to (\ref{eq2}). }
\end{remark}

\subsection{The zeros of
$P_n(z)$}\label{thezeros}

Throughout this subsection, we assume that the
map $\psi$ satisfies Conditions A.1 and A.2
stated in Subsection \ref{asymptoticformulas}.
Here we shall discuss some of the conclusions
that can be drawn from our previous results
regarding the location, limiting distribution and
accumulation points of the zeros of Carleman
polynomials.

Asymptotic formulas similar to (\ref{eq56}) and
(\ref{eq61}) are known to be satisfied by other
important systems of polynomials, e.g.,
polynomials orthogonal on the unit circle with
respect to certain types of weights \cite{Sza},
\cite{andrei}, \cite{andrei1}, and Faber
polynomials for domains with piecewise analytic
boundary \cite{minafaber}. The results that
follow are well-known consequences of such type
of behavior.

We start by introducing the notation and concepts
needed in analyzing the zeros of $P_n$. The letter $\mathcal{Z}$ denotes the set of
accumulation points of the zeros of Carleman
polynomials, i.e., $\mathcal{Z}$ consists of all
points $t\in \overline{\mathbb{C}}$ such that every
neighborhood of $t$ contains zeros of infinitely
many polynomials $P_n$.

We shall denote by $\nu_{n}$ the normalized
counting measure of the zeros of $P_n$, that is,
\begin{equation}\label{FabEq28}
\nu_{n}:=\frac{1}{n}\sum_{k=1}^n\delta_{z_{k,n}}\,,\quad
n=1,2,\ldots,
\end{equation}
where $z_{1,n},z_{2,n},\ldots,z_{n,n}$ are the
zeros of $P_n$ (counting multiplicities) and
$\delta_z$ denotes the unit point measure at $z$.

A sequence of measures $\left\{\upsilon_n\right\}_{n\geq 1}$ is said to converge in
the weak*-topology to the measure $\upsilon$
(symbolically, $\upsilon_n\wc \upsilon$ as
$n\to\infty$) if for every continuous function
$f$ defined on $\overline{\mathbb{C}}$,
$\lim_{n\to\infty}\int fd\upsilon_n=\int fd\upsilon$.

The \emph{equilibrium measure} $\mu_{L_\rho}$ of
$L_\rho$ is the probability measure supported on
$L_\rho$ whose value at any given Borel set
$B\subset L_\rho$ is
\begin{equation}\label{eq82}
\mu_{L_\rho}(B)=\frac{1}{2\pi\rho}\int_{B^{-1}}|dt|,\quad
B^{-1}:=\{t\in \mathbb{T}_\rho: \psi(t)\in B\}.
\end{equation}
Finally, for $\epsilon> 0$ and $t\in \mathbb{C}$,
$D_\epsilon(t)$ denotes the open disk with center
at $t$ and radius $\epsilon$.

From Theorem \ref{thm2}(a) and the maximum
modulus principle for analytic functions we see
that
\begin{equation}\label{FabEq68}
\frac{P_n(z)}{\sqrt{n+1}\,[\phi(z)]^n}=\phi'(z)+\left\{
\begin{array}{cc}
  \mathcal{O}\left(n^{-\lambda_1-1}(\rho/r)^n\right),\ & z\in\overline{\Omega}_r,\ \rho<r<1, \\
  \mathcal{O}\left(n^{-\lambda_1-1}\right),\ & z\in
  \overline{\Omega}_\rho\setminus\cup_{k=1}^s
  D_{\epsilon}(z_k),\
  \epsilon>0.
\end{array}
\right.
\end{equation}

Hence, we trivially have
\begin{cor}\label{cor2} For every $\epsilon>0$, there is
$N_\epsilon>0$ such that if $n>N_{\epsilon}$, then $P_n(z)$ has no zeros on
$\overline{\Omega}_\rho\setminus\cup_{k=1}^s D_{\epsilon}(z_k)$. In particular, $\mathcal{Z}\cap \Omega_\rho=\emptyset$.
\end{cor}

To describe the zero behavior of $P_n$ inside
$G_\rho$, it is convenient to  rewrite
(\ref{eq56}) as follows. Put
\[
\hat{A}_k:=A_ke^{i(\lambda_1+1)(\Theta_k-\Theta_1)}\,,\quad
1\leq k\leq u\,,
\]
and let $\theta_1,\theta_2,\ldots,\theta_u$ be such that
\[
e^{2\pi i\theta_k}=e^{i(\Theta_k-\Theta_1)},
\quad \theta_k\in (0,1],\quad 1\leq k\leq u\,,
\]
so that (\ref{eq56}) takes the form
\begin{equation}\label{FabEq64}
P^*_n(z)=H_n(z)+o(1)
\end{equation}
locally uniformly on $G_\rho$ as $n\to\infty$, where
\begin{equation}\label{eq75}
P_n^*(z):=\frac{-P_n(z)}{\sqrt{n+1}\binom{n}{
-\lambda_1-1}(\omega_1)^{n+1+\lambda_1}}\,,\quad
H_n(z):=\varphi'(z)\sum_{k=1}^u
\frac{\varphi'(z_k)\hat{A}_k e^{2\pi
in\theta_k}}{[\varphi(z)-\varphi(z_k)]^2}\,.
\end{equation}

The zeros of $P_n$ coincide, of course, with
those of $P^*_n$, and in view of Hurwitz theorem,
$\mathcal{Z}$ contains the zeros lying in
$G_\rho$ of those not identically zero functions
that are the uniform limit of some subsequence of
$\{H_n\}_{n\geq 0}$. We then pass to establish
the general form of any uniform limit point of
$\{H_n\}_{n\geq 0}$.

Among the numbers
$1=\theta_1,\theta_2,\ldots,\theta_u$, there is a
basis over the rationals  containing $\theta_1$
\cite[Ch. III. p. 4]{Cassels}, say $\theta_1,
\theta_2,\ldots,\theta_{u^*}$, $1\leq u^*\leq u$,
so that for every $k\in \{1,2,\ldots,u\}$, there
are unique rational numbers
$r_{k1},r_{k2},\ldots,r_{ku^*}$ with
\[
\theta_k=\sum_{j=1}^{u^*}r_{kj}\theta_j, \quad 1\leq k \leq
u.
\]
Note that $u^*=1$ if and only if all the
$\theta_k$'s are rational, and if $u^*\geq 2$,
then $\theta_2,\ldots,\theta_{u^*}$ are
irrational numbers linearly independent over the
rationals.

For every $k\in \{1,2,\ldots,u\}$, let $1\leq p_k
\leq q_k$ be the unique relatively prime integers
such that
\[
e^{2\pi i\,r_{k1}}=e^{2\pi i\,p_k/q_k},
\]
so that
\begin{equation}\label{eq79}
e^{2\pi i\theta_k}=e^{2\pi
i\,\left(\frac{p_k}{q_k}+\sum_{j=2}^{u^*}r_{kj}\theta_j\right)},
\quad 1\leq k \leq u,
\end{equation}
where in case $u^*=1$, the sum
$\sum_{j=2}^{u^*}\cdots$ above is understood to
be zero (observe that $p_1=q_1=1$, but $p_k<q_k$
for $k>1$).

Let $\mathrm{\mathbf{q}}$ be the least common
multiple of the denominators
$q_1,q_2,\ldots,q_u$, and for every
$\ell\in\{1,2,\ldots, \mathrm{\mathbf{q}}\}$, let
\[
\ell p_k=s_{k\ell}\!\!\!\mod q_k,\quad 0\leq s_{k\ell}<
q_k\,.
\]
Observe that  two $u$-tuples
$\left(s_{1\ell},s_{2\ell},\ldots,s_{u\ell}\right)$
corresponding to different values of $\ell$ are distinct.

\begin{thm} \label{thm6} The functions $f$ that are the uniform limit of some subsequence of $\{H_n\}_{n\geq 0}$
are the functions of the form
\begin{equation}\label{eq72}
f(z)=\varphi'(z)\sum_{k=1}^u
\frac{\varphi'(z_k)\hat{A}_k e^{2\pi
i\,\left(\frac{s_{k\ell}}{q_k}+\sum_{j=2}^{u^*}r_{kj}\alpha_j\right)}}{[\varphi(z)-\varphi(z_k)]^2}
\end{equation}
with $\ell\in\{1,2,\ldots, \mathrm{\mathbf{q}}\}$
and $\alpha_2,\ldots,\alpha_{u^*}$ arbitrary real
numbers. In particular, there is always such a
limit function $f$ that is not identically zero.
\end{thm}

\begin{cor}\label{cor3}
There exists a subsequence
$\{n_j\}_{j=1}^\infty\subset \mathbb{N}$ such
that $\nu_{n_j}\wc \mu_{L_\rho}$ as $j\to\infty$.
Hence, $L_\rho\subset \mathcal{Z}$.
\end{cor}

However, because the $z_k$'s are not necessarily
pairwise distinct, some of the limit functions
(\ref{eq72}) can be identically zero, which makes
Theorem \ref{thm1} insufficient to describe
$\mathcal{Z}\cap G_\rho$. A necessary condition
for this to happen is that for every $1\leq j\leq
u$,
\[
2\max_{k\,:\,z_k=z_j}|A_{k}|\leq
\sum_{k\,:\,z_k=z_j}|A_{k}|.
\]
It is not sufficient though, and whether for a
concrete instance of a curve $L_\rho$ satisfying
A.1 and A.2 there will be a limit function of the
form (\ref{eq72}) that is identically zero
ultimately depends on the specific values of the
$\theta_k$'s and can be determined, in principle,
from the general form given in (\ref{eq72}).

Let us then make the assumption that
\begin{enumerate}
\item[\textbf{A.3:}] No limit function of the form (\ref{eq72}) is identically
zero.
\end{enumerate}

Such an assumption is satisfied in a large number
of cases. For instance, if there is $k$ such that
$z_j\not= z_k$ whenever $j\not =k$, as is the
case in which $L_\rho$ is, in addition, a Jordan
curve.

\begin{cor}\label{cor4} Assume A.3
holds.
\begin{enumerate}
\item[(a)]
The point $t\in G_\rho$ also belongs to
$\mathcal{Z}$ if and only if there exist an
integer $\ell\in\{1,2,\ldots,
\mathrm{\mathbf{q}}\}$ and real numbers
$\alpha_2,\ldots,\alpha_{u^*}$ such that
\begin{equation}\label{eq73}
\sum_{k=1}^u \frac{\varphi'(z_k)\hat{A}_k e^{2\pi
i\,\left(\frac{s_{k\ell}}{q_k}+\sum_{j=2}^{u^*}r_{kj}\alpha_j\right)}}{[\varphi(t)-\varphi(z_k)]^2}=0.
\end{equation}

\item[(b)] For every compact set $E\subset G_\rho$
there is a number $N_E$ such that when $n>N_E$,
$P_n$ has at most $2(J-1)$ zeros in $E$ (counting
multiplicities), where $J$ is the number of
corners $z_k$. As a consequence, $\nu_{n}\wc
\mu_{L_\rho}$ as $n\to\infty$.
\end{enumerate}
\end{cor}

\begin{remark} \emph{Assume A.3 holds, so that by Corollary \ref{cor4}(a) we have the following.
If $z_1=z_2=\cdots=z_u$, then $\mathcal{Z}\cap
G_\rho=\emptyset$. Otherwise:
\begin{enumerate}
\item[(a)] if $u^*=1$ (i.e., all the $\theta_j$'s are
rational), then the number of points in
$\mathcal{Z}\cap G_\rho$ is finite, namely at
most $2(u-1)\mathrm{\mathbf{q}}$;
\item[(b)] if $u^*=2$, then by fixing $\ell$ and letting $\alpha_2$ vary,
equation (\ref{eq73}) can be written as
\begin{equation}\label{eq74}
g_{0,\ell}(z)+g_{1,\ell}(z)w+\cdots
+g_{u-1,\ell}(z)w^{2(u-1)}=0,\quad |z|=1,\quad
1\leq \ell\leq \mathrm{\mathbf{q}},
\end{equation}
where $w=\varphi(t)$ and the $g_{k,\ell}(z)$'s
are certain polynomials, so that if
$f_1,\ldots,f_m$ are those algebraic functions
that are a solution to at least one of the
algebraic equations (\ref{eq74}) (see e.g.,
\cite[Chap. 5]{Knopp}), then $\mathcal{Z}\cap
G_\rho$ consists of the traces left in $G_\rho$
by the curves $\varphi^{-1}\circ
f_1(\mathbb{T}_1),\ldots,\varphi^{-1}\circ
f_m(\mathbb{T}_1)$, plus possibly the preimages
by $\varphi$ of some of the solution points
corresponding to the algebraic singularities of
the $f_k$'s;
\item[(c)] if $u^*\geq 2$, then $\mathcal{Z} \cap
G_\rho$ is, in general, a two dimensional domain.
\end{enumerate}}
\end{remark}

\begin{remark}\emph{Under Assumption A.3, finer
results similar to Thm. 4 of \cite{andrei} (see
also \cite[Thms. 11.1, 11.2]{Simon}) on the
separation, distribution and speed of convergence
to $L_\rho$ of those zeros of $P_n$ that lie near
$L_\rho$ but separated from the corners can be
derived from Theorem \ref{thm2}(a). }
\end{remark}

\subsection{The case of some special
lemniscates}\label{lemniscates}

In this section we consider an example where
Theorem \ref{thm1} fails to describe the behavior
of certain subsequences of $\{P_n\}_{n\geq 0}$.
In particular, it shows that Corollary
\ref{cor4}(b) does not necessarily hold in the
absence of Condition A.3.

Let $s \geq 2$ be a given integer. If we agree in
that
\[
2\pi(k-1)\leq \arg(w^s+1)<2\pi k\quad
\mathrm{whenever}\quad 2\pi(k-1)/s\leq
\arg(w)<2\pi k/s,
\]
then the function $w\mapsto (w^s+1)^{1/s}$ maps
$\mathbb{E}_1$ conformally onto the exterior of
the lemniscate of $s$ petals $\{z:|z^s-1|=1\}$
(see Figure \ref{fig1} below for $s=3$).

Let $R>1$ be a number that has been fixed, and let
\begin{equation}\label{eq57}
L_1:=\left\{z:|z^s-1|=R^s\right\}=\left\{z=(w^s+1)^{1/s}:|w|=R\right\}.
\end{equation}
Then, for this $L_1$ we have
\[
\psi(w)=(R^sw^s+1)^{1/s},\quad
\phi(z)=R^{-1}(z^s-1)^{1/s},
\]
\[
\rho=R^{-1},\quad
\Omega_\rho=\{z:|z^s-1|>1\},\quad
L_\rho=\{z:|z^s-1|=1\},\quad
G_\rho=\{z:|z^s-1|<1\},\quad
\]
and it is easily seen that $\psi$ satisfies
Conditions A.1 and A.2 of Subsection
\ref{asymptoticformulas} with
\[
\omega_k=R^{-1}e^{i(2k-1)\pi /s},\quad z_k=0\,,
\quad 0<\lambda_k=1/s\leq 1/2,\quad
k=1,2,\ldots,s.
\]

The important feature to note of this example is
that the function $H_n$ defined in (\ref{eq75})
is identically zero for every $n\not=s-2\mod s$.

\begin{thm}\label{thm4} Let $\left\{P_n\right\}_{n=0}^\infty$ be
 the sequence of polynomials orthonormal over the interior of the lemniscate
 $L_1=\left\{z:|z^s-1|=R^s\right\}$. Then,
\begin{enumerate}
\item[(a)] for all $n=sm+s-1$ with $m\geq 0$ an integer,
\begin{equation*}
P_n(z)=\sqrt{n+1}\,R^{-(n+1)}z^{s-1}(z^s-1)^m;
\end{equation*}

\item[(b)] for all $n=sm+l$ with $m\geq 0$ and $0\leq l\leq s-2$ integers,
\begin{equation*}
\frac{(-1)^m R^{n+1}\Gamma\left(
n+(3s-l-1)/s\right)P_{n}(z)}{n!\,\sqrt{n+1}} =
\frac{s^{(2s-l-1)/s}}{(s-l-1)!\,\Gamma((1+l-s)/s)}\frac{\partial^{s-l-2}L}{\partial
\zeta^{s-l-2}}(0,z)+R_n(z),\quad z\in G_\rho,
\end{equation*}
where $R_n(z)=\mathcal{O}(n^{-1})$ locally
uniformly on $G_\rho$ as $n\to\infty$,
$n\not=s-1\mod s$. More precisely,
\begin{equation*}
nR_n(z)=\frac{s^{(s-l-1)/s}}{\Gamma((1+l-2s)/s)}\left[
\frac{(s-1)}{2\,(s-l-2)!}\frac{\partial^{s-l-2}L}{\partial
\zeta^{s-l-2}}(0,z)-\frac{s^2}{(2s-l-2)!}\frac{\partial^{2s-l-2}L}{\partial
\zeta^{2s-l-2}}(0,z)\right]+\mathcal{O}\left(\frac{1}{n}\right)\,.
\end{equation*}
\end{enumerate}
\end{thm}

The partial derivatives of the function
$L(\zeta,z)$ occurring in Theorem \ref{thm4} can
be explicitly computed. Indeed, it is easy to see
that
\[
\varphi(z)=\frac{Rz}{[R^{2s}-1+z^s]^{1/s}}
\]
maps $G_1$ conformally onto $\mathbb{D}_1$. We
have $\varphi(0)=0$,
$\varphi'(0)=R(R^{2s}-1)^{-1/s}>0$, and hence
\[
L(\zeta,z)=\frac{(R^{2s}-1)^2[R^{2s}-1+z^s]^{(1-s)/s}[R^{2s}-1+\zeta^s]^{(1-s)/s}}{\left(
[R^{2s}-1+z^s]^{1/s}\zeta-z[R^{2s}-1+\zeta^s]^{1/s}\right)^2}.
\]

Moreover, it easily follows by mathematical
induction that for all $0\leq j\leq s-1$,
\[
\frac{\partial^j L(\zeta,z)}{\partial
\zeta^j}=\frac{(-1)^j(j+1)!(R^{2s}-1)^2[R^{2s}-1+z^s]^{(1-s+j)/s}[R^{2s}-1+\zeta^s]^{(1-s)/s}}{\left(
[R^{2s}-1+z^s]^{1/s}\zeta-z[R^{2s}-1+\zeta^s]^{1/s}\right)^{j+2}}+\zeta^{s-j}f_j(\zeta,z),
\]
where $f_j(\cdot,z)$ is analytic at $0$, and
therefore, for all $0\leq l \leq s-2$,
\begin{equation}\label{eq55}
\frac{\partial^{s-l-2}L}{\partial
\zeta^{s-l-2}}(0,z)=(s-l+1)!
z^{l-s}\left[\frac{R^{2s}-1}{R^{2s}-1+z^s}\right]^{(l+1)/s}\,.
\end{equation}

Thus, we obtain from Theorem \ref{thm4}(b) and
(\ref{eq55}) the following
\begin{cor}\label{cor5}For any compact set $F\subset G_\rho$, there is a
number $N_F$ such that if $n>N_F$ and $n\not=s-1
\mod s$, then $P_n$ has no zeros on $F$. As a
consequence, $\nu_n\wc \mu_{L_\rho}$ as
$n\to\infty$, $n\not=s-1\mod s$.
\end{cor}

However, by Theorem \ref{thm4}(a), the zeros of
$P_{sm+s-1}(z)$ are fixed, namely, a zero of
multiplicity $s-1$ at the origin and the points
$\alpha_k:=e^{2\pi ik/s}$, $1\leq k\leq s$, each
of multiplicity $m$ and contained in $G_\rho$.
Therefore,
\[
\nu_{sm+s-1}\wc
\frac{1}{s}\sum_{k=1}^s\delta_{\alpha_k}\quad \mathrm{as}\ m\to\infty.
\]
Thus, Corollary \ref{cor4}(b) does not
necessarily hold in the absence of Condition A.3.

The assertion in Corollary \ref{cor5} concerning
the convergence of the measures $\nu_n$ can be
proven by employing the exact same argument used
in the proof of Corollary \ref{cor4}(b).

\begin{figure}
\centering
\includegraphics[scale=.4]{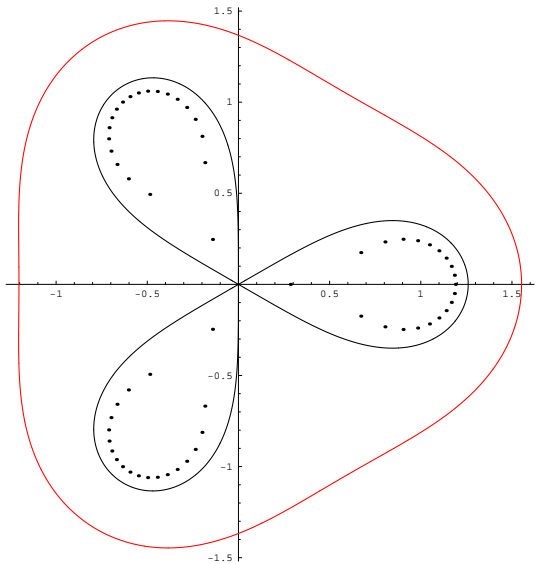}
\caption{Zeros of $P_{60}(z)$ for $L_1=\left\{z:|z^3-1|=(1.4)^3\right\}$.}\label{fig1}
\end{figure}

\section{Proofs}\label{proofs}

\subsection{Proofs of Theorem \ref{thm3} and Corollary \ref{cor1}}\label{subsection3.1}

We first make a short remark about the estimates
in (\ref{eq38}). Carleman stated his formula
(\ref{eq77}) \cite[Satz IV]{Carleman} as holding
uniformly on the exterior of the curve $L_1$ with
the estimate
$h_n(z)=\mathcal{O}\left(\sqrt{n}\rho^n\right)$.
However, as shown by Gaier in \cite[Thm. 2,
$\S$2]{Gaier1}, without any variation, Carleman's
proof equally yields that (\ref{eq77}) also holds
uniformly on any $L_r$, $\rho<r<1$, with the
estimate
$h_n(z)=\mathcal{O}\left(n^{-1/2}(\rho/r)^n\right)$.
Here we note that, indeed, from Carleman's proof
it actually follows that on any $L_r$ with
$1<r<\infty$, (\ref{eq77}) holds with the
estimate $h_n(z)=\mathcal{O}\left(\rho^n\right)$.
To see this, follow Gaier's presentation \cite[p.
13]{Gaier1} of Carleman's proof, and notice that,
in Gaier's notation,
\[
\sum_{j=1}^nj\rho^{2j-2-2n}=\left\{\begin{array}{ll}
                                                        n(n+1)/2,\  &\ \rho=1,
                                                        \\
                                                        {\displaystyle
                                                        \frac{n\rho^{2n}(\rho^2-1)+1-\rho^{2n}}{\rho^{2n}(1-\rho^2)^2}}
                                                        ,\ &\ \rho\not=1,
                                                      \end{array}
\right.\nonumber\\
\]
(beware that the meanings of $\rho$ and $r$ in
the present paper are exchanged in Gaier's
presentation) so that the quantity $C_n$ defined
in page 15 of \cite{Gaier1} indeed satisfies
\[
C_n=\left\{\begin{array}{ll}
\mathcal{O}\left(r^n\right),\  &\ \rho>1, \\
\mathcal{O}\left(\sqrt{n}r^n\right),\  &\ \rho=1, \\
\mathcal{O}\left(n^{-1/2}(r/\rho)^n\right),\  &\
r<\rho<1,
\end{array}
\right.
\]
from which one gets (\ref{eq38}) by following the
same line of argument that concludes the proof in
\cite{Gaier1}.

\paragraph{Proof of Theorem \ref{thm3}.} Equality
(\ref{eq71}) holds uniformly in $\zeta\in
\overline{G}_1$ for each fixed $z\in G_1$. We
then replace $P_k(\zeta)$ in (\ref{eq71}) by its
corresponding asymptotic representation given by
the right-hand side of (\ref{eq77}), multiply the
resulting equality by
$[\phi(\zeta)]^{-(n+1)}/2\pi i$ and integrate it
over $L_1$ to obtain
\begin{equation}\label{eq78}
\frac{\overline{\varphi'(z)}}{2\pi
i}\oint_{L_1}\frac{\varphi'(\zeta)[\phi(\zeta)]^{-(n+1)}d\zeta}
{\left[1-\overline{\varphi(z)}\varphi(\zeta)\right]^2}=\sum_{k=0}^\infty
\frac{\overline{P_k(z)}\sqrt{k+1}}{2\pi
i}\oint_{L_1}\phi'(\zeta)[\phi(\zeta)]^{k-n-1}[1+h_k(\zeta)]\,d\zeta.
\end{equation}

On the one hand, we have \[
d\zeta=\frac{i\varphi(\zeta)|\varphi'(\zeta)||d\zeta|}{\varphi'(\zeta)}\Rightarrow
\overline{d\zeta}
=-\frac{\varphi'(\zeta)d\zeta}{\overline{\varphi'(\zeta)}[\varphi(\zeta)]^2},
\]
so that
\begin{equation}\label{eq6}
\overline{\oint_{L_1}\frac{\varphi'(\zeta)[\phi(\zeta)]^{-(n+1)}d\zeta}
{\left[1-\overline{\varphi(z)}\varphi(\zeta)\right]^2}}=-
\oint_{L_1}\frac{\varphi'(\zeta)[\phi(\zeta)]^{n+1}d\zeta}
{\left[\varphi(\zeta)-\varphi(z)\right]^2}.
\end{equation}

On the other hand,
\[
\frac{1}{2\pi
i}\oint_{L_1}\phi'(\zeta)[\phi(\zeta)]^{k-n-1}\,d\zeta=\frac{1}{2\pi
i}\oint_{\mathbb{T}_1}t^{k-n-1}\,dt=\left\{\begin{array}{ll}
                                             0,\ & \mathrm{if}\ n\not=k, \\
                                             1,\ & \mathrm{if}\ n=k,
                                           \end{array}
\right. \quad k,n\geq 0,
\]
and since $h_k(\psi(w))$ is analytic in
$\mathbb{E}_\rho$,
\[
\frac{1}{2\pi
i}\oint_{L_1}\phi'(\zeta)[\phi(\zeta)]^{k-n-1}h_k(\zeta)\,d\zeta=\frac{1}{2\pi
i}\oint_{\mathbb{T}_1}t^{k-n-1}h_k(\psi(t))\,dt=0,
\quad \quad 0\leq k\leq n-1.
\]
Therefore, we get from (\ref{eq78}),
(\ref{eq6}) and the two previous relations that
\begin{equation}\label{eq5}
\frac{\varphi'(z)}{\sqrt{n+1}\,2\pi
i}\oint_{L_1}\frac{\varphi'(\zeta)[\phi(\zeta)]^{n+1}d\zeta}
{\left[\varphi(\zeta)-\varphi(z)\right]^2}=
P_n(z) -\epsilon_n(z),\qquad z\in G_1,
\end{equation}
where
\begin{equation}\label{eq3}
\epsilon_n(z)=-\sum_{j=0}^\infty\sqrt{1+j/(n+1)}\,\alpha_{j,n}\,P_{n+j}(z),\qquad
z\in G_1,
\end{equation}
with
\[
\alpha_{j,n}:=\overline{\frac{1}{2\pi
i}\oint_{L_1}\phi'(\zeta)[\phi(\zeta)]^{j-1}h_{n+j}(\zeta)\,d\zeta\,}.
\]

Now, from (\ref{eq38}) we obtain that for every
$1<\eta<\infty$, there is a constant $M_\eta$
that only depends on $\eta$, such that for all
integers $j,\,n\geq 0$,
\begin{equation}\label{eq39}
\left|\alpha_{j,n}\right| = \left|\frac{1}{2\pi
i}\oint_{L_\eta}\phi'(\zeta)[\phi(\zeta)]^{j-1}h_{n+j}(\zeta)\,d\zeta\right|\leq
M_\eta \eta^{j}\rho^{n+j}.
\end{equation}
By Carleman's formula, for every $r\in
(\rho,1/\rho)$, there is some constant $K_r$ such
that $|P_n(z)|\leq K_r\sqrt{n+1}r^n$ for all
$z\in\overline{G}_{r}$, so that for $1<\eta<(\rho
r)^{-1}$
\begin{eqnarray*}
\sum_{j=0}^\infty\sqrt{1+j/(n+1)}\,\left|\overline{\alpha_{j,n}}\,P_{n+j}(z)\right|&\leq&
\sqrt{n+1}(r\rho)^nM_\eta
K_r\sum_{j=0}^\infty\left(1+\frac{j}{n+1}\right)
(\eta r\rho)^j\\
&\leq & \frac{\sqrt{n+1}(r\rho)^nM_\eta
K_r}{\left(1-\eta r\rho\right)^2}\,,\quad z\in
\overline{G}_r\,.
\end{eqnarray*}
This shows that the series in the right-hand side
of (\ref{eq3}) that defines $\epsilon_n(z)$ for
$z\in G_1$, indeed converges locally uniformly on
$G_{1/\rho}$ to an analytic function (the
analytic continuation of $\epsilon_n(z)$). In the
same way one sees that if for some
$\tau\in[0,1/\rho)$ and constant $K_\tau$,
$|P_n(z)|\leq K_\tau\sqrt{n+1} \tau^n$ for all
$z\in E\subset G_{1/\rho}$, then for any fixed
$1<\eta<(\rho \tau)^{-1}$,
\[
|\epsilon_n(z)|\leq
\frac{\sqrt{n+1}(\tau\rho)^nM_\eta
K_\tau}{\left(1-\eta
\tau\rho\right)^2}\quad\forall\, z\in E.
\]
The proof of Theorem \ref{thm3} is complete.

\paragraph{Proof of Corollary \ref{cor1}.} Equality (\ref{eq28}) holds
for any given conformal map $\varphi$ of $G_1$
onto $\mathbb{D}_1$. Let us pick specifically a
map $\varphi$ such that $\varphi(z)\not=0$ for
all $z\in G_1\cap\Omega_\rho$. Such a map has an
analytic and univalent continuation to all of
$G_{1/\rho}$ (given by (\ref{eq66}) and
(\ref{eq70})), so that for arbitrary $\eta\in
(1,1/\rho)$,
\begin{equation}\label{eq40}
P_{n}(z)= \frac{\sqrt{n+1}\,\varphi'(z)}{2\pi
 i} \oint_{L_\eta}
 \frac{\phi'(\zeta)[\phi(\zeta)]^{n}d\zeta}{\varphi(\zeta)-\varphi(z)}+\epsilon_n(z),\quad
 z\in G_1.
\end{equation}
The right-hand side of (\ref{eq40}) is indeed a
well-defined analytic function (in the variable
$z$) on $G_{\eta}$, so that by analytic
continuation, (\ref{eq40}) actually holds for all
$z\in G_{\eta}$. Now, for every $z\in L_1$, the
function (in the variable $\zeta$)
$\varphi'(z)/\left[\varphi(\zeta)-\varphi(z)\right]$
is analytic at all points of
$G_{1/\rho}\cap\Omega_\rho$, with the exception
of the point $z$, where it has a simple pole of
residue $1$. From the residue theorem, it follows
that for every $r\in (\rho,1)$,
\begin{equation}\label{eq41}
P_{n}(z)=\sqrt{n+1}\phi'(z)[\phi(z)]^{n}
+\frac{\sqrt{n+1}\,\varphi'(z)}{2\pi
 i} \oint_{L_r}
 \frac{\phi'(\zeta)[\phi(\zeta)]^{n}d\zeta}{\varphi(\zeta)-\varphi(z)}+\epsilon_n(z),\quad z\in
 L_1.
\end{equation}
Now, fix $r'\in (\rho,1)$, so that $L_{r}$ is contained in the interior of $L_{r'}$  whenever $\rho<r<r'$, and therefore
\begin{eqnarray}\label{eq42}
\left|\frac{1}{2\pi
 i} \oint_{L_r}
 \frac{\phi'(\zeta)[\phi(\zeta)]^{n}d\zeta}{\varphi(\zeta)-\varphi(z)}\right|&=&\left|\frac{1}{2\pi
 i} \oint_{\mathbb{T}_r}
 \frac{t^{n}dt}{\varphi(\psi(t))-\varphi(z)}\right|
 \leq\frac{r^{n+1}}{\dist\left(\mathbb{T}_1,\varphi(L_r)\right)}\nonumber\\ &\leq&\frac{r^{n+1}}{\dist\left(\mathbb{T}_1,\varphi(L_{r'})\right)}, \quad \rho<r<r',\quad z\in L_1.
\end{eqnarray}
Since by Theorem \ref{thm3}, we have
\[
P_n(z)=\mathcal{O}\left(\sqrt{n}\right)\Rightarrow
e_n(z)=\mathcal{O}\left(\sqrt{n}\rho^n\right),\quad
z\in L_1,
\]
it follows from (\ref{eq41}) and (\ref{eq42}) by letting $r\to \rho$ that
\begin{equation*}
P_{n}(z)=\sqrt{n+1}\phi'(z)[\phi(z)]^{n}
+\mathcal{O}\left(\sqrt{n}\rho^n\right)=\sqrt{n+1}\phi'(z)[\phi(z)]^{n}\left[1
+\mathcal{O}\left(\rho^n\right)\right],\quad z\in L_1.
\end{equation*}
The proof of Corollary \ref{cor1} is complete.

\subsection{Asymptotic expansions}\label{lehmanexpansion}
Through the remaining of this paper, we assume
that the map $\psi$ satisfies Conditions A.1 and
A.2 stated in Subsection
\ref{asymptoticformulas}.

Let $T_k$ be a small open circular arc of
$\mathbb{T}_\rho$ centered at $\omega_k$ such
that
$\overline{T_k}\cap\{\omega_1,\ldots,\omega_s\}=\{\omega_k\}$.
The set $T_k\setminus\{\omega_k\}$ consists of
two circular arcs, say $T_k^+$, $T_k^-$, and by
our Assumption A.1 on $L_\rho$, there exist
simple analytic arcs
$\mathcal{L}_k^+\supset\psi\left(\overline{T_k^+}\right)$
and
$\mathcal{L}_k^-\supset\psi\left(\overline{T_k^-}\right)$
of which $z_k$ is an interior point. Hence, the
map $\psi$, originally defined on
$\mathbb{E}_\rho$, can be continued by the
Schwarz reflection principle for analytic arcs
\cite{Davis} across both $T_k^+$ and $T_k^-$.
Since the images of $\mathcal{L}_k^+$ and
$\mathcal{L}_k^-$ in such reflections are again
simple analytic arcs containing $z_k$ as an
interior point, by applying subsequent
reflections we can continue $\psi$ near
$\omega_k$ onto the entire logarithmic Riemann
surface $\mathcal{S}_{\omega_k}$ with branch
point at $\omega_k$.

Let the functions $(w-\omega_k)^{l+j\lambda_k}$,
$l\geq 0$, $j\geq 1$, and $\log(w-\omega_k)$ be
defined in $\mathcal{S}_{\omega_k}$. In what follows
we abbreviate by putting $y=w-\omega_k$. Lehman
\cite[Thm. 1]{Lehman} proved that when
$\lambda_k>0$, $\psi$ has the following
asymptotic expansion about $\omega_k$: if
$\lambda_k$ is \emph{irrational}, then
\begin{equation}\label{FabEq1}
\psi(w)=\psi(\omega_k)+\sum_{l= 0}^\infty \sum_{j=1}^\infty
c^{k}_{lj0}y^{l+j\lambda_k} ,\quad c^{k}_{010}\not=0\,;
\end{equation}
if $\lambda_k=p/q$ is a \emph{fraction} reduced to lowest terms,
then
\begin{equation}\label{FabEq2}
\psi(w)=\psi(\omega_k)+\sum_{l=0}^\infty
\sum_{j=1}^q \sum_{m=0}^{\lfloor
l/p\rfloor}c^{k}_{ljm}y^{l+j\lambda_k}(\log
y)^m,\quad c^{k}_{010}\not=0.
\end{equation}

The terms in the above series are assumed to be arranged in an order
such that a term of the form $y^{l+j\lambda_k}(\log y)^m$ precedes
one of the form $y^{l'+j'\lambda_k}(\log y)^{m'}$ if either
$l+j\lambda_k<l'+j'\lambda_k$ or $l+j\lambda_k=l'+j'\lambda_k$ and
$m>m'$. We write in (\ref{FabEq1}) $c^k_{lj0}$ instead of
simply $c^k_{lj}$ when $\lambda_k$ is irrational,
because this will allow us to express many of the
relations that follow in one single statement
without having to distinguish between $\lambda_k$
being irrational or rational.

The precise meaning of these expansions is the
following: if according to the order explained
above, (\ref{FabEq1}) and (\ref{FabEq2}) are
written in the form
\[
\psi(w)=\psi(\omega_k)+\sum_{n=1}^\infty \chi_n(y),
\]
then for all $N\geq 1$,
\[
\psi(w)-\psi(\omega_k)-\sum_{n=1}^N \chi_n(y)=o\left(\chi_N(y)\right)
\]
as $w\to \omega_k$ ($y\to 0$) from any finite sector
$\vartheta_1\leq\arg(w-\omega_k)\leq \vartheta_2$ of
$\mathcal{S}_{\omega_k}$.

The coefficients $c^k_{ljm}$ in (\ref{FabEq1})
and (\ref{FabEq2}) depend on the values assigned
to the functions $(w-\omega_k)^{l+j\lambda_k}$,
$\log(w-\omega_k)$ at a specified point of
$\mathcal{S}_{\omega_k}$. We shall assume that the
values of $\psi$ in $\mathbb{E}_\rho$  define
$\psi$ in the sector
$\Theta_k-\pi<\arg(w-\omega_k)<\Theta_k+\pi$ of
$\mathcal{S}_{\omega_k}$, and that for every $w$ in
this sector,
\[
(y)^{l+j\lambda_k}=|y|^{l+j\lambda_k}e^{i(l+j\lambda_k)\arg(y)},\quad
\log y=\log|y|+i\arg(y), \quad y=w-\omega_k.
\]

Let
\[
A_k:=c^{k}_{010}=\lim_{\underset{w\in\mathbb{E}_\rho}{w\to
\omega_k}}\frac{\psi(w)-\psi(\omega_k)}{(w-\omega_k)^{\lambda_k}}
\ (\not=0), \quad 1\leq k\leq s,
\]
so that the following relations follow from (\ref{FabEq1}) and (\ref{FabEq2}). \emph{If} $1/2<\lambda_k<1$, then
\begin{equation}\label{eq20}
\psi(w)=\psi(\omega_k)+A_k y^{\lambda_k}+c^{k}_{020}y^{2\lambda_k}+c^{k}_{110}y^{1+\lambda_k}+
\mathcal{O}\left(y^{3\lambda_k}\right);
\end{equation}
\emph{if} $\lambda_k=1/2$, then
\begin{equation}\label{eq21}
\psi(w)=\psi(\omega_k)+A_k y^{\lambda_k}+c^{k}_{020}y^{2\lambda_k}+c^{k}_{111}y^{1+\lambda_k}\log y+
\mathcal{O}\left(y^{3\lambda_k}\right);
\end{equation}
\emph{if} $0<\lambda_k<1/2$ and $\upsilon_k$ is sufficiently small (say, $0<\upsilon_k<\min\{\lambda_k,1-2\lambda_k\}$), then
\begin{equation}\label{eq22}
\psi(w)=\psi(\omega_k)+A_k y^{\lambda_k}+c^{k}_{020}y^{2\lambda_k}+c^{k}_{030}y^{3\lambda_k}+
o\left(y^{3\lambda_k+\upsilon_k}\right);
\end{equation}
\emph{if $1<\lambda_k<2$}, then
\begin{equation}\label{eq23}
\psi(w)=\psi(\omega_k)+A_k
y^{\lambda_k}+c^{k}_{110}y^{1+\lambda_k}+c^{k}_{020}y^{2\lambda_k}+
\mathcal{O}\left(y^{2+\lambda_k}\right)
\end{equation}
(notice that if $1<\lambda_k=p/q<2$, then $p\geq
3$, $q\geq 2$, and no $\log$-terms correspond to
$l=0,1,2$).

We analyze now the behavior of
$\varphi'(z)/\left[\varphi(\psi(w))-\varphi(z)\right]$
as $w\to\omega_k$, where $\varphi$ is a conformal
map of $G_1$ onto $\mathbb{D}_1$. For given
$\delta>0$ and $t\in\mathbb{C}$, we put
\[
D_\delta(t):=\{w:|w-t|<\delta\}.
\]

We have already noticed in the introduction that
the kernel
\[
L(\zeta,z):=
\frac{\varphi'(\zeta)\varphi'(z)}{[\varphi(\zeta)-\varphi(z)]^2}, \quad \zeta,\,z\in G_1,
\]
is independent of the map $\varphi$. For fixed
$z\in G_1$, $L(\cdot,z)$ is analytic on
$G_1\setminus\{z\}$. Hence, if $\epsilon>0$ is
such that $D_\epsilon(z_k)\subset G_1$, then from
Taylor's inequality we find that for all
$\zeta\in D_\epsilon(z_k)$, $z\in G_1\setminus
D_\epsilon(z_k)$, and integer $N\geq 1$,
\begin{equation}\label{eq52}
\frac{\varphi'(z)}{\varphi(\zeta)-\varphi(z)}=\frac{\varphi'(z)}{\varphi(z_k)-\varphi(z)}
-\sum_{j=0}^{N-1}\frac{\partial^{j}
L}{\partial\zeta^j}
(z_k,z)\frac{(\zeta-z_k)^{j+1}}{(j+1)!}+R_{N}(\zeta,z_k,z),
\end{equation}
with
\[
|R_{N}(\zeta,z_k,z)|\leq
\frac{2|\zeta-z_k|^{N+1}}{(N+1)!}\max\left\{\left|\frac{\partial^N
L}{\partial x^N}(x,z)\right|:|x-z_k|\leq r, \
z\in G_1\setminus D_\epsilon(z_k)\right\},\quad
|\zeta-z_k|\leq r<\epsilon\,.
\]
Combining this for $N=2$ with (\ref{eq20}), (\ref{eq21}),
(\ref{eq22}) and (\ref{eq23}) gives the
following: \emph{if $0<\lambda_k<1$}, then
\begin{eqnarray}\label{eq10}
\frac{\varphi'(z)}{\varphi(\psi(w))-\varphi(z)}&=&\frac{\varphi'(z)}{\varphi(z_k)-\varphi(z)}
-L(z_k,z)A_ky^{\lambda_k}-\left[
L(z_k,z)c^{k}_{020}+\frac{\partial
L}{\partial\zeta}
(z_k,z)\frac{(A_k)^2}{2}\right]y^{2\lambda_k}\nonumber\\
&&-\left\{\begin{array}{ll}
            \mathcal{O}\left(y^{1+\lambda_k}\right),\  &\ 1/2<\lambda_k<1,  \\
            \mathcal{O}\left(y^{3\lambda_k}\right),\  &\ 0<\lambda_k<1/2,
            \\
            L(z_k,z)c^k_{111}y^{3\lambda_k}\log y+\mathcal{O}(y^{3\lambda_k}),\  &\
            \lambda_k=1/2,
          \end{array}
\right.
\end{eqnarray}
while \emph{for $1<\lambda_k<2$},
\begin{equation}\label{eq11}
\frac{\varphi'(z)}{\varphi(\psi(w))-\varphi(z)}=\frac{\varphi'(z)}{\varphi(z_k)-\varphi(z)}
-L(z_k,z)A_ky^{\lambda_k}-
L(z_k,z)c^{k}_{110}y^{1+\lambda_k}+\mathcal{O}\left(y^{2\lambda_k}\right).
\end{equation}
These relations hold uniformly in $z$ on compact
subsets of $G_1\setminus\{z_k\}$ as $w\to \omega_k$
($y\to 0$) from any finite sector
$\vartheta_1\leq\arg(w-\omega_k)\leq \vartheta_2$ of
$\mathcal{S}_{\omega_k}$. That is, if $\delta>0$ is
so small that
\[\left\{\psi(w):
\vartheta_1\leq\arg(w-\omega_k)\leq \vartheta_2,\
|w-\omega_k|\leq \delta\right\}\subset
D_\epsilon(z_k),
\] then (\ref{eq10}) and
(\ref{eq11}) hold uniformly for $z\in
G_1\setminus D_\epsilon(z_k)$ and $w$ satisfying
$\vartheta_1\leq\arg(w-\omega_k)\leq \vartheta_2$,
$|w-\omega_k|\leq \delta$.

Now, given $k\in\{1,2,\ldots,s\}$, the Laurent expansion of $L(\cdot,z_k)$ about $z_k$ has the form
\[
L(\zeta,z_k)=\frac{1}{(\zeta-z_k)^2}+a_{k,0}+a_{k,1}(\zeta-z_k)+a_{k,2}(\zeta-z_k)^2+\cdots,
\]
with
\[
a_{k,0}=\frac{2\varphi'(z_k)\varphi'''(z_k)-3\left[\varphi''(z_k)\right]^2}{12\left[\varphi'(z_k)\right]^2},
\]
so that
\begin{equation}\label{eq13}
\frac{\varphi'(z_k)}{\varphi(\psi(w))-\varphi(z_k)}=\frac{1}{\psi(w)-z_k}
-
\frac{\varphi''(z_k)}{2\varphi'(z_k)}-\sum_{j=0}^\infty
\frac{a_{k,j}(\psi(w)-z_k)^{j+1}}{j+1},
\end{equation}
and we obtain from (\ref{eq13}), (\ref{eq20}), (\ref{eq21}), (\ref{eq22}) and (\ref{eq23}) that \emph{if $ 1/2<\lambda_k<1$,}
\begin{equation}\label{eq24}
\frac{\varphi'(z_k)}{\varphi(\psi(w))-\varphi(z_k)}=\frac{1}{A_ky^{\lambda_k}}-\frac{\varphi''(z_k)}{2\varphi'(z_k)}-\frac{c^k_{020}}{(A_k)^2}
-\frac{c^k_{110}y^{1-\lambda_k}}{(A_k)^2}+\mathcal{O}(y^{\lambda_k}),
\end{equation}
\emph{if $0<\lambda_k<1/2$} and  $\upsilon_k$ is sufficiently small (say, $0<\upsilon_k<\min\{\lambda_k,1-2\lambda_k\}$), then
\begin{equation}\label{eq25}
\frac{\varphi'(z_k)}{\varphi(\psi(w))-\varphi(z_k)}=\frac{1}{A_ky^{\lambda_k}}-\frac{\varphi''(z_k)}{2\varphi'(z_k)}-\frac{c^k_{020}}{(A_k)^2}
+\left[\frac{(c^k_{020})^2-A_kc^k_{030}-a_{k,0}(A_k)^4}{(A_k)^3}\right]y^{\lambda_k}
+\mathcal{O}(y^{\lambda_k+\upsilon_k}),
\end{equation} \emph{if $\lambda_k=1/2$},
\begin{equation}\label{eq26}
\frac{\varphi'(z_k)}{\varphi(\psi(w))-\varphi(z_k)}
=\frac{1}{A_ky^{\lambda_k}}-\frac{\varphi''(z_k)}{2\varphi'(z_k)}-\frac{c^k_{020}}{(A_k)^2}
+\left[\frac{(c^k_{020})^2-a_{k,0}(A_k)^4}{(A_k)^3}\right]y^{\lambda_k}+\mathcal{O}\left(y^{\lambda_k}\log
y\right),
\end{equation}
and finally, \emph{if $1<\lambda_k<2$}, then
\begin{equation}\label{eq27}
\frac{\varphi'(z_k)}{\varphi(\psi(w))-\varphi(z_k)}
=\frac{1}{A_ky^{\lambda_k}}-\frac{\varphi''(z_k)}{2\varphi'(z_k)}-\frac{c^k_{110}y^{1-\lambda_k}}{(A_k)^2}
+\mathcal{O}\left(1\right).
\end{equation}
Relations (\ref{eq24}), (\ref{eq25}),
(\ref{eq26}) and (\ref{eq27}) hold uniformly as
$w\to \omega_k$ from any finite sector
$\vartheta_1\leq\arg(w-\omega_k)\leq \vartheta_2$
of $\mathcal{S}_{\omega_k}$.

\subsection{Auxiliary lemmas}

Recall we are using the notation
\[
\mathbb{T}_r:=\{w:|w|=r\},\quad
\mathbb{D}_r:=\{w:|w|<r\},\quad
\mathbb{E}_r:=\{w:r<|w|\leq \infty\}.
\]

For $\delta\in (0,\rho)$ and $1\leq k\leq s$, we
have also already defined
\[
\Sigma_{\delta,k}:=\left\{w:\rho-\delta<|w|<\rho^2/(\rho-\delta),\
\Theta_k-\delta<\arg(w)<\Theta_k+\delta \right\}.
\]
Note that $\Sigma_{\delta,k}$ is the reflection
of itself about the circle $\mathbb{T}_\rho$. Let
\[
T^+_{\delta,k}:=\left\{w\in \mathbb{T}_\rho:\
\Theta_k\leq \arg(w)<\Theta_k+\delta
\right\},\quad T^-_{\delta,k}:=\left\{w\in
\mathbb{T}_\rho:\ \Theta_k-\delta< \arg(w)\leq
\Theta_k \right\},
\]
and
\[
\Sigma^+_{\delta,k}:=\Sigma_{\delta,k}\setminus
T^+_{\delta,k},\quad
\Sigma^-_{\delta,k}:=\Sigma_{\delta,k}\setminus
T^-_{\delta,k}.
\]

For all $\delta>0$ sufficiently small, the
mapping $\psi$ has analytic continuations
$\psi_+$, $\psi_-$ from $\mathbb{E}_\rho$ to
$\Sigma^+_{\delta,k}$, $\Sigma^-_{\delta,k}$,
respectively. To see this, fix a number
$\delta'>0$ such that
$\overline{T^-_{\delta'\!,k}}\,\cap
\,\{\omega_1,\ldots,\omega_s\}=\{\omega_k\}$. By
Assumption A.1, there is a simple analytic arc
$\mathcal{L}^-_k$, of which $z_k$ is an interior
point, such that $\mathcal{L}^-_k\supset
\psi\left(T^-_{\delta',k}\right)$. Let $z^*$
denote the Schwarz reflection of $z$ about
$\mathcal{L}^-_k$ (see \cite{Davis}), which is
well-defined in some neighborhood $U^k$ of
$\mathcal{L}^-_k$. Then, for every $\delta\in
(0,\delta')$ so small that
$\psi\left(\overline{\mathbb{E}}_\rho\cap
\Sigma_{\delta,k} \right)\subset U^k$, the
analytic continuation $\psi_+$ of $\psi$ from
$\mathbb{E}_\rho$ to $\Sigma^+_{\delta,k}$ is
given by
\[
\psi_+(w)=
\left[\psi\left(\rho^2/\overline{w}\right)\right]^*,\quad
w\in \Sigma^+_{\delta,k}\cap \mathbb{D}_\rho.
\]

For $a,\,b\in \mathbb{C}$, we denote by $[a,b]$
the oriented closed segment that starts at $a$
and ends at $b$. A similar meaning is attached to
$(a,b)$, $(a,b]$ and $[a,b)$.

For every $0<\sigma<\rho$, we define
\[
\sigma_k:=\sigma \omega_k/\rho=\sigma
e^{i\Theta_k}\,,\quad 1\leq k\leq s,
\]
and the contour
\[
\Gamma_\sigma:=\mathbb{T}_\sigma\cup\left(\cup_{k=1}^s[\sigma_k,\omega_k]\right)\,.
\]
Each segment $[\sigma_k,\omega_k]$ of
$\Gamma_\sigma$ is thought of as having two
sides, and the exterior of the contour
$\Gamma_\sigma$, denoted by
$\mathrm{ext}(\Gamma_\sigma)$, is understood to
be the unbounded component of
$\overline{\mathbb{C}}\setminus\Gamma_\sigma$, that is,
\[
\mathrm{ext}(\Gamma_\sigma)=\mathbb{E}_\sigma\setminus\left(\cup_{k=1}^s[\sigma_k,\omega_k]\right)\,.
\]

\begin{lem}\label{lem1} Let $\delta\in (0,\rho)$ be such that $\Sigma_{\delta,1},\Sigma_{\delta,2},
\dots,\Sigma_{\delta,s}$ are pairwise disjoint
and $\psi$ has analytic continuations $\psi_\pm$
to $\Sigma^\pm_{\delta,k}$ for each $1\leq k\leq
s$. Then there exists $\sigma\in
(\rho-\delta,\rho)$ such that
\begin{enumerate}
\item[(a)] $\psi$ has an analytic continuation from
$\mathbb{E}_\rho$ to $\ext(\Gamma_\sigma)$ with
continuous boundary values on $\Gamma_\sigma$
when viewing each $[\sigma_k,\omega_k]$ as having
two sides;

\item[(b)] $\psi$ is one-to-one on $
A_{\delta,\sigma}:=\overline{\mathbb{E}}_\sigma\setminus
\cup_{k= 1}^s \overline{\Sigma_{\delta,k}}$ and
\[
\psi\left(A_{\delta,\sigma}\right)\cap\overline{\psi\left(\ext(\Gamma_\sigma)\setminus
A_{\delta,\sigma}\right)}  =\emptyset.
\]
\end{enumerate}
\end{lem}

\begin{proof} We first need to introduce the
following notation. The set
$\mathbb{T}_\rho\setminus
\{\omega_1,\omega_2,\ldots,\omega_s\}$ consists of $s$
open circular arcs $\ell^1,\ell^2,\ldots,\ell^s$.
Let $k\in\{1,2,\ldots,s\}$,  and choose numbers
$\alpha_1<\alpha_2$ (which depend on $k$) such
that
\[
\ell^k=\left\{\rho e^{i\alpha}: \alpha_1< \alpha
< \alpha_2\right\}.
\]
Then, for any $v>0$ and $\sigma\in (0,\rho)$, we
define
\[
\ell^k_{v}:=\left\{\rho e^{i\alpha}:
\alpha_1+v\leq \alpha\leq \alpha_2-v\right\},
\]
\[
O^k_{v,\sigma}:=\left\{w:\sigma<|w|<\rho^2/\sigma,\
\alpha_1+v<\arg(w)<\alpha_2-v\right\}.
\]
Observe that $\ell^k_0=\overline{\ell^k}$, and that
$O^k_{v,\sigma}$ is an open set whose reflection
about $\mathbb{T}_\rho$ coincides with itself.

Let $\delta_1>0$ be so small that each of the
arcs
$\ell^1_{\delta_1},\ell^2_{\delta_1},\ldots,\ell^s_{\delta_1}$
has positive length. We first prove the following

\emph{Claim: for all $\sigma\in
(\rho-\delta_1,\rho)$ sufficiently close to
$\rho$, $\psi$ has an analytic continuation from
$\mathbb{E}_\rho$ to $\mathbb{E}_\rho\cup
\left[\cup_{k=1}^s O^k_{\delta_1,\sigma}\right] $
which is one-to-one on $\cup_{k=1}^s
O^k_{\delta_1,\sigma}$.}\vspace{.2cm}

In effect, for each $1\leq k\leq s$, the Schwarz
reflection $z\mapsto z^*$ about the simple
analytic arc $\psi\left(\ell^k_{\delta_1}\right)$
is a well-defined (one-to-one and antianalytic)
function on some small neighborhood $U^k$ of
$\psi\left(\ell^k_{\delta_1}\right)$,  so that
for all $\sigma$ so close to $\rho$ that
$\psi\left(\overline{\mathbb{E}}_\rho\cap
O^k_{\delta_1,\sigma}\right)\subset U^k$, the
analytic continuation of $\psi$ to
$O^k_{\delta_1,\sigma}$ is given by
\[
\psi(w)=\left[\psi(\rho^2/\overline{w})\right]^*,\quad
w\in O^k_{\delta_1,\sigma}\cap \mathbb{D}_\rho.
\]
By Assumption A.1, $\psi$ is one-to-one on
$\overline{\mathbb{E}}_\rho\setminus
\{\omega_1,\omega_2,\ldots,\omega_s\}$, and being
the closed arcs
$\ell^1_{\delta_1},\ell^2_{\delta_1},\ldots,\ell^s_{\delta_1}$
pairwise disjoint, it clearly follows that if
$\rho-\sigma$ is sufficiently small, then $\psi$
is univalent on $\cup_{k=1}^s
O^k_{\delta_1,\sigma}$. Thus, the claim is
proven.

Now, fix $\delta\in (0,\rho)$ such that
$\Sigma_{\delta,1},\Sigma_{\delta,2},
\dots,\Sigma_{\delta,s}$ are pairwise disjoint
and $\psi$ has analytic continuations $\psi_+$,
$\psi_-$ to $\Sigma^+_{\delta,k}$,
$\Sigma^-_{\delta,k}$, respectively, for each
$1\leq k\leq s$. By applying the claim above for
an arbitrarily small value of $\delta'\in
(0,\delta)$, it clearly follows that for
$\rho-\sigma$ sufficiently small, $\psi$ has an
analytic continuation from $\mathbb{E}_\rho$ to
$\ext(\Gamma_\sigma)$ with continuous boundary
values on $\Gamma_\sigma$ when viewing each
$[\sigma_k,\omega_k]$ as having two sides.
Furthermore, every point $w_0\in
\mathbb{T}_\rho\setminus\{\omega_1,\omega_2,\ldots,\omega_s\}$
has a neighborhood on which $\psi$ is one-to-one.
This proves Lemma \ref{lem1}(a).

Suppose now that the statement of Lemma
\ref{lem1}(b) is not true, that is, suppose there
is a sequence $\{\sigma_n\}_{n\geq 1}$, with
$\sigma_n \nearrow \rho$, such that either $\psi$
is not one-to-one on
\[
A_{\delta,\sigma_n}:=\overline{\mathbb{E}_{\sigma_n}}\setminus
\cup_{k= 1}^s\overline{\Sigma_{\delta,k}},
\]
or
\[
\psi\left(A_{\delta,\sigma_n}
\right)\cap\overline{\psi\left(\ext(\Gamma_{\sigma_n})\setminus
A_{\delta,\sigma_n}\right)} \not=\emptyset.
\]

Then we can find two sequences of points
$\{w_{n,0}\}_{n\geq1}$, $\{w_{n,1}\}_{n\geq1}$,
such that for each $n\geq 1$,
\[
w_{n,0}\not=w_{n,1},\quad w_{n,0}\in
A_{\delta,\sigma_n},\quad w_{n,1}\in
\overline{\mathbb{E}}_{\sigma_n}
\]
and either one of the following three equalities
holds true:
\begin{equation}\label{eq48}
\psi(w_{n,0})=\psi(w_{n,1}),\quad
\psi(w_{n,0})=\psi_+(w_{n,1}),\quad
\psi(w_{n,0})=\psi_-(w_{n,1}).
\end{equation}
By extracting subsequences if necessary, we can
assume that
\[
w_{n,0}\to w_0 \in
\overline{\mathbb{E}}_\rho\setminus \cup_{k=
1}^s\Sigma_{\delta,k},\quad w_{n,1}\to w_1\in
\overline{\mathbb{E}}_\rho,
\]
and, by (\ref{eq48}), $\psi(w_0)=\psi(w_1)$. But,
in view of Assumption A.1, this is only possible
if $w_1=w_0\in
\overline{\mathbb{E}}_\rho\setminus\{\omega_1,\omega_2,\ldots,\omega_s\}$,
contradicting the fact that $\psi$ is univalent
in some neighborhood of $w_0$.

\end{proof}

\begin{lem}\label{lem2} Let $\delta>0$ be such
that $\psi$ has analytic continuations $\psi_\pm$
to $\Sigma^\pm_{\delta,k}$. Then, for every
$\epsilon>0$ and  $\sigma\in (\rho-\delta,\rho)$
such that
\[
\psi_\pm\left([\sigma_k,\omega_k]\right)\subset
D_\epsilon(z_k) \subset G_1,
\]
we have
\begin{eqnarray*}
&&\hspace{-.5cm}\frac{1}{2\pi
i}\int_{\sigma_k}^{\omega_k}\left(\frac{\varphi'(z)t^n}{\varphi(\psi_+(t))-\varphi(z)}-
\frac{\varphi'(z)t^n}{\varphi(\psi_-(t))-\varphi(z)}\right)dt\\
&=&-\binom{n}{-\lambda_k-1}\rho^{n+\lambda_k+1}\left(L(z_k,z)A_ke^{i(n+\lambda_k+1)\Theta_k}
+r_{\sigma_k,n}(z)
 \right)
\end{eqnarray*}
where
\[
r_{\sigma_k,n}(z)=\left\{\begin{array}{ll}
                      \mathcal{O}\left(n^{-\lambda_k}\right),\ &\ if\  0<\lambda_k<1,\ \lambda_k\not=1/2,\\
                      \mathcal{O}\left(n^{-1}\log n\right),\ &\ if\ \lambda_k=1/2, \\
                      \mathcal{O}\left(n^{-1}\right),\ &\ if\ 1<\lambda_k<2, \\
                      \end{array}\right.
\]
uniformly on $G_1\setminus D_\epsilon(z_k)$ as
$n\to\infty$.
\end{lem}
\begin{proof} For the analytic functions $(w-\omega_k)^\beta$,
$\log(w-\omega_k)$, defined on
$\Sigma_{\delta,k}\cap\mathbb{E}_\rho$ and
corresponding to the branch of the argument
\[\Theta_k-\pi<\arg(w-\omega_k)<\Theta_k+\pi,\quad
w\in\mathbb{C}\setminus\{t\omega_k:t\leq 1\},
\]
let us denote by $(w-\omega_k)^\beta_\pm$ and
$\log_\pm(w-\omega_k)$ their analytic continuations
from $\Sigma_{\delta,k}\cap\mathbb{E}_\rho$ onto
$\Sigma^\pm_{\delta,k}$, respectively. For an
integer $n\geq 0$ and real $\beta>-1$, we have
(see, e.g., \cite[Sec. 2]{minafaber})
\[
\int_0^1x^n(1-x)^\beta\log (1-x)
dx=-\frac{\Gamma(\beta+1)n!\log
n\left[1+\mathcal{O}\left(1/\log
n\right)\right]}{\Gamma(n+\beta+2)}\quad
(n\to\infty),
\]
\[
\int_0^1x^n(1-x)^\beta
dx=\frac{\Gamma(\beta+1)n!}{\Gamma(n+\beta+2)}=
\frac{\Gamma(\beta+1)(1+o(1))}{n^{\beta+1}}\quad
(n\to\infty),
\]
so that
\begin{eqnarray*}
&&\hspace{-.5cm}\int_{\sigma_k}^{\omega_k}t^n(t-\omega_k)_\pm^\beta\log_\pm(t-\omega_k)dt\nonumber\\
&=&e^{\mp i\beta\pi}
\omega_k^{n+1+\beta}\int_{0}^{1}x^n(1-x)^\beta
\left[\log(1-x)+\log\rho +i(\Theta_k\mp\pi)\right]dx+\mathcal{O}(\sigma^n) \nonumber \\
&=& -\frac{e^{\mp i\beta\pi}
\omega_k^{n+1+\beta}\Gamma(\beta+1)n!(\log
n)}{\Gamma(n+\beta+2)}
+\mathcal{O}\left(\frac{\rho^n}{n^{\beta+1}}\right)+\mathcal{O}(\sigma^n).
\end{eqnarray*}
From this last equality  and the well-known
identity
\[
\Gamma(1-z)\Gamma(z)=-\Gamma(-z)\Gamma(z+1)=\pi/\sin(\pi
z),
\]
we then obtain
\begin{eqnarray}\label{eq30}
&&\int_{\sigma_k}^{\omega_k}t^n\left[(t-\omega_k)_+^\beta\log_+(t-\omega_k)-
(t-\omega_k)_-^\beta\log_-(t-\omega_k)\right]dt\nonumber\\
&=& -\frac{2\pi i n!(\log n)(\omega_k)^{n+1+\beta}}
{\Gamma(-\beta)\Gamma(n+\beta+2)}
+\mathcal{O}\left(\frac{\rho^n}{n^{\beta+1}}\right)+\mathcal{O}(\sigma^n)\nonumber\\
&=& -2\pi i \binom{n}{-\beta-1}(\log
n)(\omega_k)^{n+1+\beta}
+\mathcal{O}\left(\frac{\rho^n}{n^{\beta+1}}\right).
\end{eqnarray}

Similarly,
\begin{equation}\label{eq31}
\int_{\sigma_k}^{\omega_k}
t^n\left[(t-\omega_k)_+^\beta-(t-\omega_k)_-^\beta\right]
dt=2\pi i \binom{n}{-\beta-1}(\omega_k)^{n+1+\beta}+
\mathcal{O}\left(\sigma^n\right),
\end{equation}
and
\begin{equation}\label{eq32}
\int_{\sigma_k}^{\omega_k}\mathcal{O}\left((t-\omega_k)_\pm^{\beta}
t^n\right)dt=\mathcal{O}\left(\frac{\rho^n}{n^{1+\beta}}\right).
\end{equation}
Therefore, we get from (\ref{eq10}),
(\ref{eq30})-(\ref{eq32}) that if
$0<\lambda_k<1$, $\lambda_k\not=1/2$, then
\begin{eqnarray}\label{eq1}
&&\hspace{-.5cm}\frac{1}{2\pi
i}\int_{\sigma_k}^{\omega_k}\left(\frac{\varphi'(z)t^n}{\varphi(\psi_+(t))-\varphi(z)}-
\frac{\varphi'(z)t^n}{\varphi(\psi_-(t))-\varphi(z)}\right)dt\nonumber\\
&=&-\binom{n}{-\lambda_k-1}L(z_k,z)A_k(\omega_k)^{n+\lambda_k+1}
-\binom{n}{-2\lambda_k-1}\left[
L(z_k,z)c^{k}_{020}+\frac{\partial
L}{\partial\zeta} (z_k,z)\frac{(A_k)^2}{2}\right]
(\omega_k)^{n+2\lambda_k+1}
\nonumber\\
& &+\left\{\begin{array}{ll}
            \mathcal{O}\left(\rho^n/n^{2+\lambda_k}\right), \ &\ 1/2<\lambda_k<1, \\
            \mathcal{O}\left(\rho^n/n^{3\lambda_k+1}\right), \ &\ 0<\lambda_k<1/2,
          \end{array}\right.
         \nonumber \\
&=&-\binom{n}{-\lambda_k-1}\left(L(z_k,z)A_k(\omega_k)^{n+\lambda_k+1}
+\frac{\left[ L(z_k,z)c^{k}_{020}+\frac{\partial
L}{\partial\zeta}
(z_k,z)\frac{(A_k)^2}{2}\right]\Gamma(-\lambda_k)
(\omega_k)^{n+2\lambda_k+1}}{\Gamma(-2\lambda_k)n^{\lambda_k}}\right.
\nonumber\\
&&\hspace{2.4cm}\left.+o\left(\frac{\rho^n}{n^{\lambda_k}}\right)\right),
\end{eqnarray}
while if $\lambda_k=1/2$, then
\begin{eqnarray}\label{eq14}
&&\hspace{-.5cm}\frac{1}{2\pi
i}\int_{\sigma_k}^{\omega_k}\left(\frac{\varphi'(z)t^n}{\varphi(\psi_+(t))-\varphi(z)}-
\frac{\varphi'(z)t^n}{\varphi(\psi_-(t))-\varphi(z)}\right)dt\nonumber\\
&=&-\binom{n}{-\lambda_k-1}\left(L(z_k,z)A_k(\omega_k)^{n+\lambda_k+1}-\frac{L(z_k,z)c^k_{111}\Gamma(-\lambda_k)(\omega_k)^{n+3\lambda_k+1}(\log
n)}{\Gamma(-3\lambda_k)n^{2\lambda_k}}+o\left(\frac{\rho^n\log
n}{n^{2\lambda_k}}\right)\right).
\end{eqnarray}
Similarly, if $1<\lambda_k<2$, then we obtain
from (\ref{eq11}) and (\ref{eq31})-(\ref{eq32})
that
\begin{eqnarray}\label{eq15}
&&\hspace{-.5cm}\frac{1}{2\pi
i}\int_{\sigma_k}^{\omega_k}\left(\frac{\varphi'(z)t^n}{\varphi(\psi_+(t))-\varphi(z)}-
\frac{\varphi'(z)t^n}{\varphi(\psi_-(t))-\varphi(z)}\right)dt\nonumber\\
&=&-\binom{n}{-\lambda_k-1}\left(L(z_k,z)A_k\omega_k^{n+\lambda_k+1}+\frac{L(z_k,z)c^k_{110}
\Gamma(-\lambda_k)\omega_k^{n+\lambda_k+2}}{\Gamma(-\lambda_k-1)n}+o\left(\frac{\rho^n}{n}\right)\right).
\end{eqnarray}
\end{proof}
Lemma \ref{lem2} is nothing but an abbreviation
of relations (\ref{eq1}), (\ref{eq14}) and
(\ref{eq15}).

\subsection{Proof of Theorems \ref{thm1} and \ref{thm2}}

\paragraph{Proof of Theorems \ref{thm1}.}
Let $\delta\in (0,\rho)$ be such that
$\Sigma_{\delta,1},\Sigma_{\delta,2},
\dots,\Sigma_{\delta,s}$ are pairwise disjoint
and $\psi$ has analytic continuations $\psi_\pm$
to $\Sigma^\pm_{\delta,k}$ for each
$k\in\{1,2,\ldots,s\}$. Let  $F\subset G_\rho$ be
a compact set, so that by (\ref{eq28}) we have
that for any (fixed) $r\in(\rho,1)$,
\begin{equation}\label{eq85}
P_{n}(z)=\frac{\sqrt{n+1}}{2\pi
 i}\oint_{\mathbb{T}_1}\frac{\varphi'(z)t^{n}dt}{\varphi(\psi(t))-\varphi(z)}
+\mathcal{O}\left(\sqrt{n}(r\rho)^n\right),\quad
z\in  E.
\end{equation}

Now, it is clear that we can find $\epsilon>0$
and $\sigma\in (\rho-\delta,\rho)$ such that for
every $k\in \{1,2,\ldots,s\}$,
\begin{equation}\label{eq84}
\psi_\pm\left([\sigma_k,\omega_k]\right)\subset
D_\epsilon(z_k)\subset G_1\setminus F.
\end{equation}
In addition, we can assume that $\sigma$ was
chosen so close to $\rho$ that it satisfies the
thesis of Lemma \ref{lem1}(a) and that
\begin{equation}\label{eq87}
\overline{\psi\left(\ext(\Gamma_\sigma)\cap
\mathbb{D}_1\right)}\subset
\overline{G}_1\setminus F.
\end{equation}

This last inclusion implies that for each $z\in
F$, the function
$\left[\varphi(\psi(w))-\varphi(z)\right]^{-1}$
is analytic on
$\mathbb{D}_1\cap\ext(\Gamma_\sigma)$, with
continuous boundary values on $\mathbb{T}_1\cup
\Gamma_\sigma$ when viewing each
$[\sigma_k,\omega_k]$ as having two sides. Hence,
by deforming the path of integration in
(\ref{eq85}) from $\mathbb{T}_1$ to
$\Gamma_\sigma$, and taking into account
(\ref{eq84}), Lemma \ref{lem2} and (\ref{eq87}),
we obtain that for all $z\in F$,
\begin{eqnarray*}\label{eq86}
P_{n}(z)&=&\frac{\sqrt{n+1}}{2\pi
 i}\oint_{\mathbb{T}_\sigma}\frac{\varphi'(z)t^{n}dt}{\varphi(\psi(t))-\varphi(z)}
+\mathcal{O}\left(\sqrt{n}(r\rho)^n\right)\\
&&+\sqrt{n+1}\sum_{k=1}^s\frac{1}{2\pi
i}\int_{\sigma_k}^{\omega_k}\left(\frac{\varphi'(z)t^n}{\varphi(\psi_+(t))-\varphi(z)}-
\frac{\varphi'(z)t^n}{\varphi(\psi_-(t))-\varphi(z)}\right)dt\nonumber\\
&=&\mathcal{O}\left(\sqrt{n}(r\rho)^n\right)+\mathcal{O}\left(\sqrt{n}\sigma^n\right)-\sqrt{n+1}\sum_{k=1}^s
\binom{n}{-\lambda_k-1}\rho^{n+\lambda_k+1}\left(L(z_k,z)A_ke^{i(n+\lambda_k+1)\Theta_k}
+r_{\sigma_k,n}(z)\right)\\
&=&-
\sqrt{n+1}\binom{n}{-\lambda_1-1}\rho^{n+\lambda_1+1}\left(\sum_{k=1}^u
L(z_k,z)A_ke^{i(n+\lambda_k+1)\Theta_k}+R_n(z)\right)
\end{eqnarray*}
where
\begin{equation}\label{eq89}
R_n(z)=
 \left\{\begin{array}{ll}
                      \mathcal{O}\left(n^{-\lambda_1}\right),\ &\mathrm{if}\  0<\lambda_1<1,\ \lambda_1\not=1/2,\\
                      \mathcal{O}\left(n^{-1}\log n\right),\ &\mathrm{if}\ \lambda_1=1/2, \\
                      \mathcal{O}\left(n^{-1}\right),\ &\mathrm{if}\ 1<\lambda_1<2, \\
                      \end{array}\right.+
\left\{\begin{array}{ll}
\mathcal{O}\left(n^{\lambda_1-\lambda_{u+1}}\right),\
&\mathrm{if}\ u<s, \\
0,\ &\mathrm{if}\ u=s,
\end{array}\right.
\end{equation}
uniformly on $F$ as $n\to\infty$. This proves
Theorem \ref{thm1}.

\paragraph{Proof of Theorems \ref{thm2}.} Let  $\delta\in
(0,\rho)$ be such that
$\Sigma_{\delta,1},\Sigma_{\delta,2},
\dots,\Sigma_{\delta,s}$ are pairwise disjoint
and  $\psi$ has analytic continuations $\psi_\pm$
to $\Sigma^\pm_{\delta,k}$ for each
$k\in\{1,2,\ldots,s\}$. Since, by Assumption A.1,
$\psi\left(\overline{\mathbb{E}}_\rho\setminus\cup_{k=1}^s\Sigma_{\delta,k}\right)$
is a closed set that does not contain any corner
$z_k$, we can find $\epsilon>0$ and $\sigma\in
(\rho-\delta,\rho)$ such that for every $k\in
\{1,2,\ldots,s\}$,
\begin{equation}\label{eq44}
\psi_\pm\left([\sigma_k,\omega_k]\right)\subset
D_\epsilon(z_k)\subset G_1,\quad
D_\epsilon(z_k)\cap
\psi\left(\mathbb{E}_\sigma\setminus\cup_{k=1}^s\overline{\Sigma_{\delta,k}}\right)=\emptyset.
\end{equation}

In addition, of course, we can assume that
$\sigma$ satisfies the thesis of Lemma
\ref{lem1}, so that for each $z\in
V_{\delta,\sigma}:=\psi\left(\mathbb{E}_\sigma\setminus
\cup_{k= 1}^s
\overline{\Sigma_{\delta,k}}\right)$, the
function
$\left[\varphi(\psi(w))-\varphi(z)\right]^{-1}$
is analytic on
$\left[\mathbb{D}_1\cap\ext(\Gamma_\sigma)\right]\setminus\{\phi(z)\}$,
with continuous boundary values on
$\mathbb{T}_1\cup \Gamma_\sigma$ and a simple
pole at $\phi(z)$. Then, we get from (\ref{eq28})
and the residue theorem that
\begin{eqnarray}\label{eq47}
P_{n}(z)&=&\frac{\sqrt{n+1}}{2\pi
 i}\oint_{\mathbb{T}_1}\frac{\varphi'(z)t^{n}dt}{\varphi(\psi(t))-\varphi(z)}
+\epsilon_n(z)\nonumber\\
&=&\sqrt{n+1}\phi'(z)[\phi(z)]^n+\frac{\sqrt{n+1}}{2\pi
i}\oint_{\mathbb{T}_\sigma}\frac{\varphi'(z)t^n
dt}{\varphi(\psi(t))-\varphi(z)}+\epsilon_n(z)\nonumber\\
&&+\sqrt{n+1}\sum_{k=1}^s\frac{1}{2\pi
i}\int_{\sigma_k}^{\omega_k}\left(\frac{\varphi'(z)t^n}{\varphi(\psi_+(t))-\varphi(z)}-
\frac{\varphi'(z)t^n}{\varphi(\psi_-(t))-\varphi(z)}\right)dt,\quad
z\in V_{\delta,\sigma}.
\end{eqnarray}

Let then $E\subset V_{\delta,\sigma}$ be a
compact set and let $r\in (\rho, 1)$ be such that
$E\subset \overline{G}_{r}$, so that
$\epsilon_n(z)=\mathcal{O}\left(\sqrt{n}(r\rho)^n\right)$
uniformly on $E$ as $n\to\infty$. On the other
hand, again by Lemma \ref{lem1}, the set
\[
\psi(\Gamma_\sigma):=\overline{\psi\left(\ext(\Gamma_\sigma)\right)}\setminus
\psi\left(\ext(\Gamma_\sigma)\right)
\]
is compact and disjoint from $E$, so that
\begin{equation}\label{eq88}
\min\left\{|\varphi(\zeta)-\varphi(z)|:\zeta\in
\psi(\Gamma_\sigma),\ z\in E\right\}>0.
\end{equation}

Therefore, we get from (\ref{eq47}) by
considering (\ref{eq44}), Lemma \ref{lem2} and
 (\ref{eq88}) that for all $z\in E$,
\begin{eqnarray*}
P_n(z) &=&
\sqrt{n+1}\phi'(z)[\phi(z)]^n+\mathcal{O}\left(\sqrt{n}\sigma^n\right)+\mathcal{O}\left(\sqrt{n}(r\rho)^n\right)
\\
&&-\sqrt{n+1}\sum_{k=1}^s
\binom{n}{-\lambda_k-1}\rho^{n+\lambda_k+1}\left(L(z_k,z)A_ke^{i(n+\lambda_k+1)\Theta_k}
+r_{\sigma_k,n}(z)\right)\\
&=&\sqrt{n+1}\phi'(z)[\phi(z)]^n-
\sqrt{n+1}\binom{n}{-\lambda_1-1}\rho^{n+\lambda_1+1}\left(\sum_{k=1}^u
L(z_k,z)A_ke^{i(n+\lambda_k+1)\Theta_k}+R_n(z)\right),
\end{eqnarray*}
where $R_n(z)$ satisfies (\ref{eq89}) uniformly
on $E$ as $n\to\infty$.

Thus, it only remains to prove Theorem
\ref{thm2}(b). Let $\delta'>0$ and $\epsilon>0$
be such that
$\Sigma_{\delta'\!,1},\Sigma_{\delta'\!,2},
\dots,\Sigma_{\delta'\!,s}$ are pairwise
disjoint,  $\psi$ has analytic continuations
$\psi_\pm$ to $\Sigma^\pm_{\delta'\!,k}$ for each
$k\in\{1,2,\ldots,s\}$,
\begin{equation}\label{eq9}
\overline{\psi_\pm\left(\Sigma^\pm_{\delta'\!,k}\right)}\subset
D_\epsilon(z_k) \subset G_1,\quad 1\leq k\leq s,
\end{equation}
and
\begin{equation}\label{eq50}
z_k\not=z_j\Rightarrow D_\epsilon(z_k)\cap
D_\epsilon(z_j)=\emptyset.
\end{equation}

Since for every $1\leq k\leq s$,
\[
\lim_{w\to\omega_k}\frac{\psi_{\pm}(w)-z_k}{(w-\omega_k)^{\lambda_k}}=A_k>0,
\]
we can assume $\delta'$ was chosen so small that
\begin{equation}\label{eq49} z_k\not\in\psi_\pm
\left(\Sigma^\pm_{\delta'\!,k}\right),\quad 1\leq
k\leq s.
\end{equation}

Then, fix $\delta\in (0,\delta')$, and for this
$\delta$, choose $\sigma\in (0,\rho)$ satisfying
Lemma \ref{lem1}. Let $z_j$ be a (fixed) corner
of $L_\rho$, and for each $k$ such that
$\psi(\omega_k)=z_j$, let $\ell_k\subset
\Sigma_{\delta,k}\cap\overline{\mathbb{E}}_\rho$
be a positively oriented closed simple path
encircling the segment $(\sigma_k,\omega_k]$,
whose only common point with $\Gamma_\sigma$ is
$\sigma_k$.

By Lemma \ref{lem1}, (\ref{eq9}), (\ref{eq50})
and (\ref{eq49}), the function
$\left[\varphi(\psi(w))-\varphi(z_j)\right]^{-1}$
is analytic on
$\mathbb{D}_1\cap\ext(\Gamma_\sigma)$, with
continuous boundary values on $\mathbb{T}_1\cup
\Gamma_\sigma\setminus\{\omega_k:\psi(\omega_k)=z_j\}$.
Hence, we obtain from (\ref{eq28}) that (with
$\rho<r<1$)
\begin{eqnarray}\label{eq34}
P_n(z_j)&=&\sum_{k:z_k\not=z_j}\frac{\sqrt{n+1}}{2\pi
i}\int_{\sigma_k}^{\omega_k}\left(\frac{\varphi'(z_j)t^n}{\varphi(\psi_+(t))-\varphi(z_j)}-
\frac{\varphi'(z_j)t^n}{\varphi(\psi_-(t))-\varphi(z_j)}\right)dt\nonumber\\
&&+\sum_{k:z_k=z_j}\frac{\sqrt{n+1}}{2\pi
i}\oint_{\ell_{k}}\frac{\varphi'(z_j)t^{n}dt}{\varphi(\psi(t))-\varphi(z_j)}
+\mathcal{O}(\sqrt{n}\sigma^n)+\mathcal{O}\left(\sqrt{n}(r\rho)^n\right)\,.
\end{eqnarray}

Now, after integrating by parts a couple of times
over $\ell_k$  and using (\ref{eq31}) we get
\begin{eqnarray}\label{eq33}
\oint_{\ell_k}t^n(t-\omega_k)^{-\lambda_k}dt
&=&\frac{(\sigma_k)^n\left[(\sigma_k-\omega_k)_-^{1-\lambda_k}-(\sigma_k-\omega_k)_+^{1-\lambda_k}\right]}{(1-\lambda_k)}\nonumber\\
&&-\frac{n(\sigma_k)^{n-1}\left[(\sigma_k-\omega_k)_-^{2-\lambda_k}-(\sigma_k-\omega_k)_+^{2-\lambda_k}\right]}
{(1-\lambda_k)(2-\lambda_k)}\nonumber\\
&&
+\frac{n(n-1)}{(1-\lambda_k)(2-\lambda_k)}\oint_{\ell_k}t^{n-2}(t-\omega_k)^{2-\lambda_k}dt
\nonumber\\
&=&\frac{n(n-1)}{(1-\lambda_k)(2-\lambda_k)}\int_{\sigma_k}^{\omega_k}t^{n-2}\left[(t-\omega_k)_+^{2-\lambda_k}-(t-\omega_k)_-^{2-\lambda_k}
\right]dt +\mathcal{O}(n\sigma^n)\nonumber\\
&=&\frac{2\pi i
n!(\omega_k)^{n-\lambda_k+1}}{\Gamma(\lambda_k)\Gamma(n+2-\lambda_k)}+\mathcal{O}(n^2\sigma^n).
\end{eqnarray}

Then, combining (\ref{eq34}), Lemma \ref{lem2},
relations (\ref{eq24})-(\ref{eq27}) and
(\ref{eq33}), we conclude that
\begin{eqnarray}
\frac{P_n(z_j)}{\sqrt{n+1}}&=&-\sum_{k:z_k\not=z_j}\binom{n}{-\lambda_k-1}\rho^{n+\lambda_k+1}\left(L(z_k,z_j)A_ke^{i(n+\lambda_k+1)\Theta_k}
+o(1)\right)\nonumber\\
&&+\sum_{k:z_k=z_j}\left[\binom{n}{\lambda_k-1}(A_k)^{-1}(\omega_k)^{n-\lambda_k+1}
+\left\{\begin{array}{ll}
\mathcal{O}\left(\rho^n/n^{1+\lambda_k}\right),\ & 0<\lambda_k\leq 1/2, \\
\mathcal{O}\left(\rho^n/n^{2-\lambda_k}\right),\
&  1/2<\lambda_k<2,\ \lambda_k\not=1,
\end{array}\right.
\right]\nonumber\\
&&+\mathcal{O}(\sigma^n)+\mathcal{O}\left((r\rho)^n\right)\nonumber,
\end{eqnarray}
so that if $\lambda_j^*=\max\{\lambda_k:
z_k=z_j,\ 1\leq k\leq s\}$, then
\begin{eqnarray}\label{eq59}
\frac{P_n(z_j)}{\sqrt{n+1}\binom{n}{\lambda_j^*-1}}-\sum_{\underset{\lambda_k=\lambda_j^*}{k:z_k=z_j}
}(A_k)^{-1}(\omega_k)^{n-\lambda_j^*+1}&=&\left\{\begin{array}{ll}
\mathcal{O}\left(\rho^n/n^{2\lambda_j^*}\right),\ & 0<\lambda_j^*\leq 1/2, \\
\mathcal{O}\left(\rho^n/n\right),\ &
1/2<\lambda_j^*<2,\ \lambda_j^*\not=1,
\end{array}\right.\nonumber\\
&&
+\sum_{\underset{\lambda_k<\lambda_j^*}{k:z_k=z_j}
}\mathcal{O}\left(\frac{\rho^n}{n^{\lambda_j^*-\lambda_k}}\right)+
\sum_{k:z_k\not=z_j}\mathcal{O}\left(\frac{\rho^n}{n^{\lambda_k+\lambda_j^*}}\right),
\end{eqnarray}
that is Theorem \ref{thm2}(b) holds true.

\subsection{Proof of Theorem \ref{thm4}}

\emph{Proof of part (a)}: Fix $m\geq 0$ an
integer, and let us show that $z^{s-1}(z^s-1)^m$
is orthogonal over $G_1$ to all powers $z^n$ with
$0\leq n\leq sm+s-2$ an integer. For this, first
notice that since both $G_1$ and the area measure
are invariant under a rotation of angle $e^{2\pi
i/s}$, we have that for any two integers
$\alpha$, $\beta$ such that $\alpha-\beta\not=
0\mod s$,
\[
\int_{G_1}z^\alpha\overline{z^\beta}dA(z)=0\,,
\]
and therefore,  $z^{s-1}(z^s-1)^m=\sum_{j=0}^m \binom{m}{j}(-1)^{m-j}z^{sj+s-1}$ is orthogonal over $G_1$ to all powers $z^{sm'+l}$ with $0\leq m'\leq m$, $0\leq l\leq s-2$ integers.

For the remaining powers, we obtain by applying
Green's formula (\cite[p. 10]{Gaier1}) and making
the change of variables $z=(R^sw^s+1)^{1/s}$ that
for $0\leq m'\leq m$,
\begin{eqnarray*}
\frac{1}{\pi}\int_{G_1}z^{s-1}(z^s-1)^m\overline{z^{sm'+s-1}}dA(z)&=&\frac{1}{s(m'+1)2\pi i}\oint_{L_1}z^{s-1}(z^s-1)^m\overline{z^{s(m'+1)}}dz\\
&=&\frac{R^{s(m+1)}}{s(m'+1)2\pi i}\oint_{\mathbb{T}_1}w^{sm+s-1}\overline{(R^sw^s+1)^{m'+1}}dw\\
&=&\left\{\begin{array}{ll}
            0,\ &\ 0\leq m'\leq m-1, \\
            R^{2s(m+1)}/(sm+s),\ &\ m'=m\,.
          \end{array}
\right.
\end{eqnarray*}

\emph{Proof of part (b):} We have that
$\rho=R^{-1}$, so that as shown in the proof of
Theorem \ref{thm1}, if $F\subset G_\rho$ is
compact, then for $r\in (R^{-1},1)$ and certain
$\sigma\in (0,R^{-1})$ sufficiently close to
$R^{-1}$, we have that uniformly in $z\in F$ as
$n\to\infty$,
\begin{equation}\label{eq51}
P_n(z)=\sum_{k=1}^s\frac{\sqrt{n+1}}{2\pi
i}\int_{\sigma_k}^{\omega_k}\left(\frac{\varphi'(z)t^n}{\varphi(\psi_+(t))-\varphi(z)}-
\frac{\varphi'(z)t^n}{\varphi(\psi_-(t))-\varphi(z)}\right)dt
+\mathcal{O}\left(\sqrt{n}(r/R)^n\right)+\mathcal{O}\left(\sqrt{n}\sigma^n\right).
\end{equation}

Now, suppose $n=sm+l$, with $0\leq l\leq s-2$,
$m\geq 0$ integers. From (\ref{eq52}) for
$N=3s-l-2$, we get that uniformly in $w\in
\cup_{k=1}^s[\sigma_k,\omega_k]$, $z\in F$,
\[
\frac{\varphi'(z)}{\varphi(\psi_\pm(w))-\varphi(z)}=\frac{\varphi'(z)}{\varphi(0)-\varphi(z)}
-\sum_{j=0}^{3s-l-3}\frac{\partial^{j}
L}{\partial\zeta^j}
(0,z)\frac{[\psi_\pm(w)]^{j+1}}{(j+1)!}+\mathcal{O}\left([\psi_\pm(w)]^{3s-l-1}\right),
\]
and since $e^{2\pi i
k/s}\psi(w)=\psi\left(e^{2\pi i k/s}w\right)$ for
all $1\leq k\leq s$, we have in virtue of
(\ref{eq51}), (\ref{eq21}), (\ref{eq22}) and
(\ref{eq32}) that
\begin{eqnarray}\label{eq83}
P_n(z)&=&-\sum_{j=0}^{3s-l-3}\frac{\partial^{j}
L}{\partial\zeta^j} (0,z)\frac{\left(\sum_{k=1}^s
e^{2(k-1)(n+j+2)i\pi/s}\right)\sqrt{n+1}}{(j+1)!\,2\pi
i}\int_{\sigma_1}^{\omega_1}t^n\left([\psi_+(t)]^{j+1}-[\psi_-(t)]^{j+1}\right)dt
\nonumber\\
&&+\mathcal{O}\left(\sqrt{n}\,R^{-n} n^{-(4s-l-1)/s}\right)\nonumber\\
&=&-\frac{\partial^{s-l-2}
L}{\partial\zeta^{s-l-2}}
(0,z)\frac{s\sqrt{n+1}}{(s-l-1)!\,2\pi
i}\int_{\sigma_1}^{\omega_1}t^n\left([\psi_+(t)]^{s-l-1}-[\psi_-(t)]^{s-l-1}\right)dt\nonumber\\
&&-\frac{\partial^{2s-l-2}
L}{\partial\zeta^{2s-l-2}}
(0,z)\frac{s\sqrt{n+1}}{(2s-l-1)!\,2\pi
i}\int_{\sigma_1}^{\omega_1}t^n\left([\psi_+(t)]^{2s-l-1}-[\psi_-(t)]^{2s-l-1}\right)dt\nonumber\\
&&+\mathcal{O}\left(\sqrt{n}\,R^{-n}
n^{-(4s-l-1)/s}\right).
\end{eqnarray}

Now,
\begin{equation*}
\psi(w)=(R^sw^s+1)^{1/s}=(Rs)^{1/s}e^{i\pi(s-1)/s^2}(w-\omega_1)^{1/s}
\left[1+\frac{(s-1)Re^{-i\pi/s}(w-\omega_1)}{2s}+\mathcal{O}\left((w-\omega_1)^2\right)\right],
\end{equation*}
and therefore,
\begin{eqnarray}\label{eq53}
[\psi(w)]^{s-l-1}&=&(Rs)^{(s-l-1)/s}e^{i\pi(s-1)(s-l-1)/s^2}(w-\omega_1)^{(s-l-1)/s}\nonumber\\
&&\times\left[1
+\frac{(s-1)(s-l-1)Re^{-i\pi/s}(w-\omega_1)}{2s}
+\mathcal{O}\left((w-\omega_1)^{2}\right) \right]
\end{eqnarray}
and
\begin{equation}\label{eq54}
[\psi(w)]^{2s-l-1}=(Rs)^{(2s-l-1)/s}e^{i\pi(s-1)(2s-l-1)/s^2}(w-\omega_1)^{(2s-l-1)/s}\left[1
+\mathcal{O}\left(w-\omega_1\right)\right]\,.
\end{equation}
 Then, from (\ref{eq53}) and (\ref{eq54}), we
 obtain
\begin{eqnarray*}
\frac{1}{2\pi
i}\int_{\sigma_1}^{\omega_1}t^n\left([\psi_+(t)]^{s-l-1}-[\psi_-(t)]^{s-l-1}\right)dt&=&
\frac{(-1)^{m}R^{-(n+1)}s^{(s-l-1)/s}n!}{\Gamma\left(n+(3s-l-1)/s\right)\Gamma\left((1+l-s)/s\right)}\\
&&+\frac{(-1)^{m}R^{-(n+1)}s^{-(l+1)/s}(s-1)(s-l-1)n!}{2\Gamma\left(n+(4s-l-1)/s\right)
\Gamma\left((1+l-2s)/s\right)}\\
&&+\mathcal{O}\left(R^{-n}
n^{-(4s-l-1)/s}\right),
\end{eqnarray*}
\begin{eqnarray*}
\frac{1}{2\pi
i}\int_{\sigma_1}^{\omega_1}t^n\left([\psi_+(t)]^{2s-l-1}-[\psi_-(t)]^{2s-l-1}\right)dt&=&
\frac{(-1)^{m+1}R^{-(n+1)}s^{(2s-l-1)/s}n!}{\Gamma\left(n+(4s-l-1)/s\right)\Gamma\left((1+l-2s)/s\right)}\\
&&+\mathcal{O}\left(R^{-n}
n^{-(4s-l-1)/s}\right).
\end{eqnarray*}
 Theorem \ref{thm4}(b) follows immediately after inserting these two previous equalities into
(\ref{eq83}).
\subsection{Proofs of the zero results}

\paragraph{Proof of Theorem \ref{thm6}.} Suppose that for some subsequence
$\{n_\nu\}_{\nu\geq 1}\subset \mathbb{N}$,
\[
H_{n_\nu}(z)=\varphi'(z)\sum_{k=1}^u
\frac{\varphi'(z_k)\hat{A}_k e^{2\pi
in_\nu\theta_k}}{[\varphi(z)-\varphi(z_k)]^2}\to
f(z)\quad \mathrm{as}\quad \nu\to\infty.
\]
By extracting a subsequence if necessary, we may
assume that for some fixed $\ell\in
\{1,\ldots,\mathbf{q}\}$,
$n_\nu=\mathrm{\mathbf{q}}m_\nu+\ell$ with
$m_\nu\in \mathbb{N}$, and by the compacity of
$\mathbb{T}_1$, that for some real numbers
$\alpha_2,\ldots,\alpha_{u^*}$
\begin{equation*}
\lim_{\nu\to\infty}e^{2\pi i r_{kj}n_{\nu}
\theta_j}=e^{2\pi i r_{kj}\alpha_j},\quad 1\leq
k\leq u,\ 2\leq j\leq u^*,
\end{equation*}
so that by (\ref{eq79}), $f$ must have the form
(\ref{eq72}).

Conversely, we now show that given an integer
$\ell\in \{1,\ldots,\mathbf{q}\}$ and arbitrary
real numbers $\alpha_2,\ldots,\alpha_{u^*}$, it
is possible to choose a subsequence
$\{n_\nu\}_{\nu\geq 1}$ such that
\begin{equation}\label{eq80}
\lim_{\nu\to\infty}e^{2\pi i n_\nu
\theta_k}=\lim_{\nu\to\infty}e^{2\pi
i\,\left(\frac{n_\nu
p_k}{q_k}+\sum_{j=2}^{u^*}r_{kj}n_\nu\theta_j\right)}=e^{2\pi
i\,\left(\frac{\ell
p_k}{q_k}+\sum_{j=2}^{u^*}r_{kj}\alpha_j\right)}\,,\quad
 1\leq k\leq u.
\end{equation}

For this, we first observe that given arbitrary
real numbers $\chi_2,\ldots,\chi_u$, it is always
possible to find a subsequence
$\{m_\nu\}_{\nu\geq 1}\subset\mathbb{N}$ such
that
\begin{equation}\label{eq8}
\lim_{\nu\to\infty}e^{2\pi i
r_{kj}\mathrm{\mathbf{q}}m_{\nu}
\theta_j}=e^{2\pi i
r_{kj}\mathrm{\mathbf{q}}\chi_j},\quad 1\leq
k\leq u,\ 2\leq j\leq u^*.
\end{equation}
In effect, consider the set of linear forms in
the variable $x$
\[\left\{r_{kj}\mathrm{\mathbf{q}}\theta_jx: 1\leq
k\leq u,\ 2\leq j\leq u^*\right\},\] and suppose
$\beta_{kj}$, $1\leq k\leq u$, $2\leq j\leq u^*$,
are integers such that
\[
\sum_{k,j}
\beta_{kj}r_{kj}\mathrm{\mathbf{q}}\theta_jx=
x\sum_{j=2}^{u^*}\left(\sum_{k=1}^u\beta_{kj}r_{kj}\mathrm{\mathbf{q}}\right)\theta_j
\]
is a linear form whose coefficient is an integer.
Then, by the linear independence of the numbers
$1,\theta_2,\ldots\theta_{u^*}$, we must have
$\sum_{k=1}^u\beta_{kj}r_{kj}\mathrm{\mathbf{q}}=0$
for every $2\leq j\leq u^*$. Hence, for an
arbitrary collection of real numbers
$\chi_2,\ldots,\chi_{u^*}$, we have $\sum_{k,j}
\beta_{kj}r_{kj}\mathrm{\mathbf{q}}\theta_j\chi_j=0$,
and so by Kronecker's theorem \cite[Chap. III,
Thm. IV.]{Cassels}, it is possible to find a
subsequence  $\{m_\nu\}_{\nu\geq 1}$ satisfying
(\ref{eq8}).

Then, choose a subsequence $\{m_\nu\}_{\nu\geq
1}$ satisfying (\ref{eq8}) with
$\chi_j=(\alpha_j-\ell\theta_j)/\mathrm{\mathbf{q}}$,
$2\leq j\leq u^*$. Then, (\ref{eq80}) is
satisfied by the subsequence
$n_\nu:=\mathrm{\mathbf{q}}m_\nu+\ell$, $\nu\in
\mathbb{N}$.

It only remains to prove that there is a function
$f$ of the form (\ref{eq72}) that is not
identically zero. Assume without loss of
generality that the set $\{k:z_k=z_1,\ 1\leq
k\leq u\}$ consists of the numbers
$1,2,\ldots,u'$ for some $u'\leq u$. It suffices
to show that it is impossible to have
\begin{equation}\label{eq81} \sum_{k=1}^{u'}\varphi'(z_k)\hat{A}_k e^{2\pi
i\,\left(\frac{\ell
p_k}{q_k}+\sum_{j=2}^{u^*}r_{kj}\alpha_j\right)}=0,
\quad \ell\in \{1,\ldots,\mathbf{q}\},\
\alpha_2,\ldots,\alpha_{u^*}\in \mathbb{R}.
\end{equation}
Assume, on the contrary, that this is the case.
Since $\varphi'(z_1)\hat{A}_1\not=0$ and
$r_{1j}=0$ for all $2\leq j\leq u^*$, we must
obviously have $r_{kj}=0$ for all $1\leq k\leq
u'$, $2\leq j\leq u^*$, and consequently,
\[
\theta_k=\frac{p_k}{q_k}\,,\qquad
k=1,2,\ldots,u'.
\]

Let $q'\leq \mathrm{\mathbf{q}}$ be the least
common multiple of the denominators
$q_1,q_2,\ldots,q_{u'}$, and for
$k=1,2,\ldots,u'$, set $p'_k:=p_k q'/q_k$, so
that $1\leq p'_k\leq q'$, and since
$\theta_1,\theta_2,\ldots,\theta_{u'}$ are
pairwise distinct, so are the numbers
$p'_1,p'_2,\ldots,p'_{u'}$, and therefore $u'\leq
q'$. Then, by (\ref{eq81}) we must have
\[
\sum_{k=1}^{u'}\varphi'(z_k)\hat{A}_k
\left(e^{2\pi i \ell/q'}\right)^{p'_k}=0 \quad
\forall\, \ell\in\{1,\ldots, q'\}.
\]
But this homogenous system of linear equations
with unknowns $\varphi'(z_k)\hat{A}_k$, $1\leq
k\leq u'$, has only the trivial solution, since
the Vandermonde matrix
$\left(a_{l,m}\right)_{1\leq\ell,m\leq q'\ }$,
$a_{l,m}=\left(e^{2\pi i\,\ell /q'}\right)^m$, is
nonsingular. This contradicts that all the
numbers $\varphi'(z_k)\hat{A}_k $ are nonzero.

\paragraph{Proof of Corollary \ref{cor4}.} \emph{Part
(a):} This is just a straightforward consequence
of Theorem \ref{thm6} and Hurwitz's theorem,
therefore, we omit it.

\emph{Part (b)}: Suppose there is a compact set
$E\subset G_\rho$ and a subsequence
$\{n_j\}\subset\mathbb{N}$ such that $P^*_{n_j}$
has more than $2(J-1)$ zeros on $E$ counting
multiplicities, where $J$ is the number of
corners of $L_\rho$ ($J\leq s$). By Assumption
A.3 (and extracting a subsequence from $\{n_j\}$
if needed), we can assume that $\{P^*_{n_j}\}$
converges locally uniformly on $G_\rho$ to a
nonzero function of the form $R(\varphi(z))$,
where $R(w)$ is a rational function with
numerator having degree no larger than $2(J-1)$.
By Hurwitz's theorem, there is an open set
$U\supset E$ such that for all $j$ large enough,
$P^*_{n_j}$ and $R(\varphi(z))$ have the same
number of zeros on $U$, contradicting our
assumption.

We now show that $\nu_{n}\wc \mu_{L_\rho}$, for
which we use standard arguments. By Helly's
selection theorem \cite[Thm. 1.3]{Saff}, from
every subsequence of $\{\nu_{n}\}_{n\geq 1}$ it
is possible to extract another subsequence
converging in the weak*-topology to a measure
$\mu$. Thus, to finish the proof, it suffices to
show that every such limit measure $\mu$ is the
equilibrium measure $\mu_{L_\rho}$ of $L_{\rho}$.

Then, suppose $\nu_{n_j}\wc \mu$ as $j\to\infty$,
so that by Corollary \ref{cor2} and what we just
proved above, $\mu$ must be supported on
$L_\rho$. Let us denote by $U^\alpha(z)$ the
logarithmic potential of the measure $\alpha$,
that is,
\[
U^{\alpha}(z):=\int\log
\frac{1}{|z-t|}d\alpha(t).
\]
Then, we obtain from Theorem \ref{thm5} and the
fact that $\nu_{n_j}\wc \mu$ that for all
$z\in\Omega_\rho$
\[
U^{\mu}(z)=\lim_{j\to\infty}U^{\nu_{n_j}}(z)
=\lim_{j\to\infty}\frac{1}{n_j}\log\frac{\kappa_{n_j}}{|P_{n_j}(z)|}=
\log|\phi'(\infty)/\phi(z)|\,.
\]
On the other hand, it is not difficult to see
from the definition of $\mu_{L_\rho}$ in
(\ref{eq82}) that for all $z\in \Omega_\rho$,
$U^{\mu_{L_\rho}}(z)=\log|\phi'(\infty)/\phi(z)|$.
Hence, $\mu$ and $\mu_{L_\rho}$ are two measures
supported on $L_\rho$ whose logarithmic
potentials coincide in $\Omega_\rho$, which in
view of Carleson's theorem \cite[Thm. 4.13]{Saff}
forces $\mu=\mu_{L_\rho}$.

\paragraph{Proof of Corollary
\ref{cor3}.} By Theorem \ref{thm6}, there is a
subsequence $\{n_j\}\subset \mathbb{N}$ such that
$\{P^*_{n_j}\}$ converges locally uniformly on
$G_\rho$ to a nonzero function of the form
(\ref{eq72}). Then, proceeding exactly as in the
proof of Corollary \ref{cor4}(b), we find that
$\nu_{n_j}\wc \mu_{L_\rho}$ as $j\to\infty$, and
since the support of the equilibrium measure
$\mu_{L_\rho}$ is the whole of $L_\rho$, we must
have $L_\rho\subset\mathcal{Z}$.

\end{document}